\documentclass[11pt]{amsart}
\usepackage[english]{babel}
\usepackage{graphicx}
\usepackage{framed}
\usepackage[normalem]{ulem}
\usepackage{amsmath}
\usepackage{amsthm}
\usepackage[mathscr]{euscript}

\usepackage{amssymb}
\usepackage{amsfonts}
\usepackage{enumerate}
\usepackage[utf8]{inputenc}
\usepackage[top=1 in,bottom=1in, left=1 in, right=1 in]{geometry}

\newcommand{\dist}{{\rm dist}}

\newcommand{\N}{\mathbb{N}}

\newcommand{\Z}{\mathbb{Z}}

\renewcommand{\emptyset}{\varnothing}

\theoremstyle{definition}
\newtheorem{theorem}{Theorem}[section]

\newtheorem{lemma}[theorem]{Lemma}
\newtheorem{proposition}[theorem]{Proposition}
\theoremstyle{definition}
\newtheorem{definition}[theorem]{Definition}
\newtheorem{remark}[theorem]{Remark}

\begin{document}

\title[Weak Mixing]{
{Weak Mixing Transformation Which Is Shannon Orbit Equivalent to a Given Ergodic Transformation}
}
\author{James O'Quinn}

\begin{abstract}
    We prove that every ergodic transformation is Shannon orbit equivalent to a weak mixing transformation. The proof is based on the techniques introduced by Fieldsteel and Friedman to show that there is a mixing transformation for a given ergodic transformation $T$ which is, for all $a\geq1$, weak-$a$-equivalent to $T$ and, for all $b\in(0,1)$, strong-$b$-equivalent to $T$. In particular, we will adapt the construction of Fieldsteel and Friedman by which they permute the columns of each Rokhlin tower in a sequence of rapidly growing Rokhlin towers so that the corresponding cocycles converge to an orbit equivalence cocycle of $T$ such that the resulting transformation and orbit equivalence have the desired properties. In addition to this, we will demonstrate a flexible method for obtaining actions of $\Z^{2} $ which are Shannon orbit equivalent to a given ergodic transformation.
\end{abstract}
\maketitle

\begin{section}{Introduction}
Originally formulated by Kerr and Li in their 2021 paper \cite{KLSOE1}, Shannon orbit equivalence adds a restriction to the combinatorial complexity of the mappings that an orbit equivalence implements between orbits. In turn, this added condition means that this is the condition between two free probability measure preserving actions of (possibly different) finitely generated virtually Abelian groups that preserves entropy, as proved by Kerr and Li in \cite{KLSOE2}. So, while Shannon orbit equivalence does preserve entropy between transformations,  we will show, using techniques from Fieldsteel and Friedman's paper \cite{FF}, that Shannon orbit equivalence does not preserve weak mixing  between integer transformations in a rather strong sense: we can always find a weak mixing transformation which is Shannon orbit equivalent to a given ergodic transformation. Before we introduce Shannon orbit equivalence, which we will do in Section 2, we will give some background definitions and results from the theory of orbit equivalence.
\begin{definition}
Let $(X,\mathscr{A},\mu)$ be a probability space, $G$ a countable discrete group, and $G\curvearrowright^{S}(X,\mu)$ a probability measure preserving action. Given an $x\in X$, we call the set \begin{align*}
 O_{S}(x)=\{x'\in X\mid x'=S^{g}(x)\text{ for some }g\in G\}
\end{align*}
the $S$-\textbf{orbit} of $x$. Suppose we have a second probability space $(Y,\mathscr{B},\nu)$, a second group $H$, and a second probability measure preserving action 
$H\curvearrowright^{T}(Y,\nu)$. We say $S$ and $T$ are \textbf{orbit equivalent} if there exists a non-singular, bimeasurable map $\phi:X\to Y$ such that, for almost every $x\in X$, $\phi$ maps the $S$-orbit of $x$ bijectively onto the $T$-orbit of $\phi(x)$. 
\end{definition}
For the rest of this paper, our probability spaces will be atomless and Lebesgue.
\begin{remark} We will use the following more convenient characterization of orbit equivalence since it will reduce some technical baggage. Two probability measure preserving actions $G\curvearrowright^{S}(X,\mu)$ and $H\curvearrowright^{T}(Y,\nu)$ are orbit equivalent if and only if there is a third action $H\curvearrowright^{R}(X,\mu)$ which is conjugate to $T$ such that, for almost every $x\in X$, we have that $O_{T}(x)=O_{S}(x)$.
\end{remark}
The next theorem, due to Dye \cite{DYE}, implies that, for ergodic transformations, orbit equivalence erases their dynamical properties.
\begin{theorem}
If $S:X\to X$ and $T:Y\to Y$ are ergodic probability measure preserving transformations such that $S$ and $T$ are free or $X$ and $Y$ are atomless, then $S$ and $T$ are orbit equivalent.
\end{theorem}

In order to turn orbit equivalence into an interesting equivalence relation, we must place some restrictions upon it. The way this is normally accomplished is by requiring additional properties for the cocycles of the orbit equivalence, which we define below.

\begin{definition}
Given two orbit equivalent free p.m.p. actions $G\curvearrowright^{S}(X,\mu)$ and $H\curvearrowright ^{T}(X,\mu)$, we can find measurable functions  $\kappa:X\times G\to H$ and $\lambda:X\times H\to G$ where $\kappa$ is given by the formula $T^{\kappa(x,g)}(x)=S^{g}(x)$ for all $g\in G$ and for a.e. $x\in X$ and $\lambda$ is given by the relationship $S^{\lambda(x,h)}(x)=T^{h}(x)$ for all $h\in H$ and for a.e. $x\in X$. We call these the \textbf{orbit equivalence cocycles} between $T$ and $S$. In general, a \textbf{cocycle} will be any map $\alpha:X\times G\to H$ which satisfies the following two conditions:
\begin{enumerate}
    \item $\alpha(x,\cdot):G\to H$ is a bijection for a.e. $x\in X$
    \item $\alpha(x, g_{1}g_{2})=\alpha(S^{g_{2}}x,g_{1})\alpha(x,g_{2})$ for all $g_{1},g_{2}\in G$ and a.e. $x\in X$.
\end{enumerate}
\end{definition}
\begin{remark} Note that, while the essential property of a cocycle is captured by the second condition, the first condition is useful to assume because we want all of our cocycles to be invertible given a fixed point from the probability space.
\end{remark}
\begin{remark}
Given a cocycle $\alpha$ of the action $G\curvearrowright^{T} (X,\mu)$, we can define a second action $G\curvearrowright^{S} (X,\mu)$ which is orbit equivalent to $T$ as follows. We first define a function $\alpha^{-1}:X\times G \to G$ given by the formula
\begin{align*}
\alpha(x,\alpha^{-1}(x,g))=g
\end{align*}
for every $g\in G$ and almost every $x\in X$ and then defining $S$ by the formula 
\begin{align*}
S^{g}(x)=T^{\alpha^{-1}(x,g)}(x)
\end{align*}
so that $\alpha$ is a cocycle for the equivalence relation between $T$ and $S$.
\end{remark}
Perhaps the two most natural types of restricted orbit equivalence are bounded and integrable orbit equivalence, which we will now define.
\begin{definition}
Let $G\curvearrowright^{S}(X,\mu)$ and $H\curvearrowright ^{T}(X,\mu)$ be orbit equivalent probability measure preserving actions with cocycles $\kappa:X\times G\to H$ and $\lambda:X\times H\to G$. We say $S$ and $T$ are \textbf{boundedly orbit equivalent} if, for every $g\in G$ and $h\in H$, the sets $\{\kappa(g,x)\mid x\in X\}$ and $\{\lambda(h,x)\mid x\in X\}$ are finite.
\end{definition}
\begin{definition}
Let $\Z\curvearrowright^{S}(X,\mu)$ and $\Z\curvearrowright ^{T}(X,\mu)$ be orbit equivalent probability measure preserving transformations with cocycles $\kappa$ and $\lambda$. We say $S$ and $T$ are \textbf{integrably orbit equivalent} if we have that 
\begin{align*}
\int |\kappa(m,x)|d\mu(x)<\infty
\end{align*}
for all $m\in\Z$ and 
\begin{align*}
\int |\lambda(m,x)|d\mu(x)<\infty
\end{align*}
for all $m\in\Z$.
\end{definition}
\begin{remark}
Note that an orbit equivalence between transformations being bounded also implies that it is integrable.
\end{remark}
The premier result related to integrable orbit equivalence is the following theorem from 1968 due to Belinskaya \cite{BELIN}.
\begin{theorem}
Two ergodic measure preserving transformations $S$ and $T$ are integrably orbit equivalent if and only if they are flip-conjugate. That is, either $S$ is conjugate to $T$ or $S$ is conjugate to $T^{-1}$.
\end{theorem}

One of the first lines of investigation to follow for restrictive orbit equivalences is to understand where they lie on the spectrum between unrestricted orbit equivalence, which erases the dynamics of ergodic actions, and conjugation. One way to accomplish this is to see if a given type of restricted orbit equivalence preserves dynamical properties which are known conjugation invariants. One such property is mixing. One of the primordial results in this direction is the following result due to Friedman and Ornstein \cite{FO} in the realm of Kakutani equivalence.
\begin{theorem}
If $T:X\to X$ is an ergodic transformation, then there is a mixing transformation $S:X\to X$ such that $T$ is kakutani equivalent to $S$.
\end{theorem}
The construction of the mixing transformation of Friedman and Ornstein highlights the critical technique that underpins the construction of the actions in the upcoming theorems: manipulating the columns of a sequence of Rokhlin towers to define the orbits of the resulting action. \par
In general, Kakutani equivalence does not imply orbit equivalence. However, del Junco and Rudolph in \cite{DJR} noticed the following relationship between even Kakutani and restricted orbit equivalence showcased in the following theorem.
\begin{theorem}
Ergodic transformations $S$ and $T$ are evenly Kakutani equivalent if and only if there is an orbit equivalence between $S$ and $T$ with cocycle $\alpha$  such that, for almost every $x$, there is an $I(x)\subset \Z$ of density 1 such that 
\begin{align*}
\lim_{\stackrel{|v|\to\infty}{v\in I(x)}}\frac{|\alpha(x,v)-v|}{|v|}=0.
\end{align*}
\end{theorem}
Using the above theorem as a model, Fieldsteel and Friedman \cite{FF} defined the following type of restricted orbit equivalence for $\Z^{d}$ ($d>1$).
\begin{definition}
Ergodic $\Z^{d}$-actions $S$ and $T$ are $id$-\textbf{Kakutani equivalent} if there is an orbit equivalence from $S$ to $T$ with cocycle $\alpha$ such that, for almost every $x\in X$, there is a set $I(x)\subset \Z^{d}$ of density 1 satisfying 
\begin{align*}
\lim_{\stackrel{|v|\to\infty}{v\in I(x)}}\frac{|\alpha(x,v)-v|}{|v|}=0.
\end{align*}
\end{definition}
Spurred by this definition, Fieldstel and Friedman asked if results related to Kakutani equivalence of $\Z$ actions can be lifted to the types of restricted orbit equivalence of $\Z^{d}$ modeled above. As Fieldsteel and Friedman show in \cite{FF}, The Friedman-Ornstein result does transfer to this case. The varieties of restricted orbit equivalence for which the Friedman-Ornstein result holds are described as follows.
\begin{definition}
A family of sets $\{I(x)\subset \Z^{d}\}_{x\in X}$ is called a sequence of \textbf{measurable full density} if there exist an increasing sequence of measurable sets $A_{i}\subset X$ and a sequence of integers $M_{i}$ such that $\lim_{i\to\infty}\mu(A_{i})=1$ and for all $x\in X$, $I(x)=\bigcup _{i=1}^{\infty}\{v\in\Z^{d}\mid |v|_{\infty}>M_{i}\text{ and } T^{v}(x)\in A_{i}\}$. For $a>1$, we say there is a \textbf{weak}-$a$-\textbf{equivalence} between $T$ and $S$ if there is an orbit equivalence between $S$ and $T$ such that for the corresponding cocycle $\alpha$ there is a sequence $\{I(x)\}_{x\in X}$ of measurable full density such that for almost every $x$,
\begin{align*}
\lim_{\stackrel{|v|\to\infty}{v\in I(x)}}\frac{|\alpha(x,v)-v|^{a}}{|v|}=0.
\end{align*}
\end{definition}
The second variant for which Fieldsteel-Friedman shows that the Friedman-Ornstein result holds is detailed below.
\begin{definition} Let $b$ be a real number such that $0<b<1$. We say that ergodic $\Z^{d}$-actions  $S$ and $T$ of $(X,\mu)$ are \textbf{strong}-$b$-\textbf{equivalent} if there is an orbit equivalence between $S$ and $T$ such that the cocycle $\alpha$ satisfies
\begin{align*}
\lim_{|v|\to\infty}\frac{|\alpha(x,v)-v|^{b}}{\min(|v|,|\alpha(x,v)|)}=0.
\end{align*}
\end{definition}
Equipped with these definitions, Fieldsteel and Friedman in \cite{FF} prove the following theorem.
\begin{theorem}
 Let $T$ be a probability measure preserving ergodic $\Z^{d}$-action on $(X,\mu)$ for some $d\geq 1$. There is a cocycle $\alpha$ for $T$ producing a mixing action $S=T^{\alpha^{-1}}$ such that the corresponding orbit equivalence between $T$ and $S$ is, for all $a\geq 1$ , a weak-$a$-equivalence and, for all $b\in(0,1)$, a strong-$b$-equivalence.
\end{theorem}
Fieldsteel and Friedman also use the same framework for this construction to get a similar result for the bounded orbit equivalence case. This appears as Theorem 3 in their paper \cite{FF}.
\begin{theorem}
If $T$ is an erogdic $\Z^{d}$-action with $d\geq2$, then there is a cocycle $\alpha$ of $T$ which gives a mixing $T^{\alpha^{-1}}$ such that the orbit equivalence between them is  bounded.
\end{theorem}
The construction involved in this theorem, as well as the main theorem in the latter half of this paper, hinges on the fact that we can create orbit equivalence cocycles by manipulating a sequence of towers. Indeed, let $(F,B)$ be a Rokhlin tower of the action $\Z^{d}\curvearrowright^{T}(X,\mu)$ and let $\mathscr{A}$ be a partition of $B$. For each $A\in\mathscr{A}$, let $\sigma_{A}:F\to F$ be a bijection . We can define a cocycle $\alpha:X\times \Z^{d}\to Z^{d}$ by setting
\begin{align*}
\alpha(x,g)=
    \begin{cases}
        \sigma_{A}(g+f) & \text{if } T^{g}x \in T^{F}B \text{ and }x=T^{f}b\text{ for some }f\in F\text{ and } b\in A\\
        g & \text{otherwise} 
    \end{cases}
\end{align*}

One key component to making this construction work is that a given sequence of cocycles begets another cocycle when a certain limit is taken. This is the cocycle for our desired action. This requires a condition for which the cocycle value for a given point in the space is fixed for a tail of the sequence. We also require certain types of 0-1 laws to guarantee that such a condition holds for almost every point in the space. The other key component is that Fieldsteel and Friedman get this sequence of cocycles by permuting columns of corresponding Rokhlin towers by independent random translations inside disjoint blocks, thereby introducing approximate independence at each stage of the construction which will produce mixing after the limit is taken.\par 
With this context from restricted orbit equivalence theory in mind, we are ready to introduce Shannon orbit equivalence.  \par

\emph{Acknowledgements}. The author greatly appreciates the guidance and helpful insights from his PhD advisor David Kerr. The research was partially funded by the Deutsche Forschungsgemeinschaft (DFG, German Research Foundation) under Germany’s Excellence Strategy – EXC
2044 – 390685587, Mathematics Münster – Dynamics – Geometry – Structure; the Deutsche
Forschungsgemeinschaft (DFG, German Research Foundation) – Project-ID 427320536 – SFB
1442, and ERC Advanced Grant 834267 - AMAREC
\end{section}
\begin{section}{Shannon Orbit Equivalence}
While Theorem 1.17 tells us that mixing is not a bounded orbit equivalence invariant for ergodic $\Z^{d}$-actions when $d\geq2$, entropy is an invariant of bounded orbit equivalence for these actions. Indeed, Austin shows in his 2016 paper \cite{A} an even stronger result: entropy is an invariant for integrable orbit equivalence of actions of fintely generated amenable groups. So, to what extent can we generalize integrable orbit equivalence so that entropy is still an invariant for a sufficiently rich class of group actions? Perhaps the most natural generalization to look at is Shannon orbit equivalence. Shannon orbit equivalence was introduced by Kerr and Li in \cite{KLSOE1} as a generalization of both integrable orbit equivalence and bounded orbit equivalence. In that paper, the authors prove that entropy is an invariant for Shannon orbit equivalences between actions over certain classes of groups which exclude $\Z$. In a subsequent paper \cite{KLSOE2}, Kerr and Li extend the aforementioned result to actions of the integers.\par

So, the natural question is the following: does a version of Belinskaya's theorem hold for Shannon orbit equivalences between ergodic actions over the integers? This was answered in the negative in a paper by Carderi, Joseph, Le Maître, and Tessera \cite{BEL}. Indeed, the authors craft examples of ergodic transformations which satisfy a weaker type of integrable orbit equivalence, which also implies Shannon orbit equivalence, but are not flip conjugate. So, if Shannon orbit equivalence between ergodic transformations does not imply flip conjugacy, yet it still preserves entropy, what other properties does Shannon orbit equivalence preserve between ergodic transformations? Turning back to Fieldsteel and Friedman, can their method of constructing cocycles be adapted to the Shannon orbit equivalence setting to produce a mixing action which is equivalent to a given ergodic action, similar to Theorem 1.16? In this paper, we will show that a weaker version of the result of Fieldsteel and Friedman does hold in the Shannon orbit equivalence setting. Namely, we will show that for a given ergodic transformation, we can find a weak mixing transformation which is Shannon orbit equivalent to it. While there is inherent tension between the methods used to guarantee mixing by Fieldsteel and Friedman and the finite entropy bounds for the cocycle partitions required by Shannon orbit equivalence, we can finesse the introduction of asymptotic independence to produce weak mixing. In particular, while our method for obtaining Shannon orbit equivalence places a bound on the number of permutations of the orbits we introduce at each stage of the construction, mixing in the Fieldsteel-Friedman construction requires unbounded growth of the number of permutations we introduce at each stage. However, we will show that we can stagger our introduction of relative mixing phenomena at each stage of the construction to guarantee weak mixing while still controlling the number of permutations we introduce at each stage to guarantee Shannon orbit equivalence. 

Before we show our weak mixing result, we will take a slight detour to introduce a method for constructing actions of $\Z^{2}$ which are Shannon orbit equivalent to a given ergodic action of $\Z$. The significance of this construction, which is adapted from the asymptotic tower constructions of Fieldsteel and Friedman, is that it allows for a great degree of flexibility in the choice of permutations applied to each column for each tower in the sequence of towers. Indeed, we replace the block translations of Fieldsteel and Friedman by more general column permutations called column-to-tile maps which encode columns of a tower of an integer action onto squares inside $\Z^{2}$ in a way that preserves the structure of the column permutations from the previous stage. In turn, a sequence of column-to-tile maps will produce an action of $\Z^{2}$ which is Shannon orbit equivalent to the original action, as long as the growth of the heights of the towers associated to each column-to-tile map is sufficiently controlled. We believe that, through careful choice of column-to-tile maps, we can produce actions of $\Z^{2}$, or perhaps actions of any discrete countable group which can be monotiled, that satisfy a particular mixing condition and are Shannon orbit equivalent to a given ergodic transformation.

\begin{definition}
Let $(X,\mathscr{A},\mu)$ be a probability space and let $\mathscr{P}\subset\mathscr{A} $ be a countable partition of $X$. The \textbf{Shannon entropy} of $P$ is the sum $\sum _{P\in \mathscr{P}}-\mu(P)\log\mu(P)$. Given a cocycle $\alpha:X\times G\to H$ and a $g\in G$, there is a natural partition $\mathscr{P}_{\alpha,g}=\{P_{g,h}\mid h\in H\}$ where $P_{g,h}=\{x\in X\mid \alpha(x,g)=h\}$. We call these the \textbf{cocycle partitions} of $\alpha$. We say an orbit equivalence between free p.m.p. actions $T$ and $S$ with cocycles $\kappa$ and $\lambda$ is \textbf{Shannon} if the cocycle partitions $\mathscr{P}_{\kappa,g}$ and $\mathscr{P}_{\lambda,h}$ have finite Shannon entropy for all $g\in G$ and all $h\in H$.
\end{definition}
The following proposition expresses the main criteria we will use to guarantee Shannon orbit equivalence.
\begin{proposition}
Let $G\curvearrowright ^{S}(X,\mu)$ and $H\curvearrowright ^{T}(X,\mu)$ be free probability measure preserving actions and $\alpha: X\times G \to H$ a cocycle between them. Fix a $g\in G$. If the cocycle partition $\mathscr{P}_{\alpha,g}$ has the property that for every $h\in H$ there is a partition $\{D_{n}^{h}\}_{n\in\N}$ of $P_{g,h}$ such that for every $n$ there is a finite set $F_{n}\subset H$ with $D_{n}^{h}=\emptyset$ for $h\notin F_{n}$ and  we have that $\sum_{n=1}^{\infty}-\mu(\bigcup_{h\in F_{n}}D_{n}^{h})\log\frac{\mu(\bigcup_{h\in F_{n}}D_{n}^{h})}{|F_{n}|}<\infty$, then $\mathscr{P}_{\alpha,g}$ has finite Shannon entropy.
\end{proposition}
\begin{proof}
Assume we have a $g\in G$ such that the the cocycle partition $\mathscr{P}_{\alpha,g}$ has the property that for every $h\in H$ there is a partition $\{D_{n}^{h}\}_{n\in\N}$ of $P_{g,h}$ such that for every $n$ there is a finite set $F_{n}\subset H$ such that $D_{n}^{h}=\emptyset$ for $h\notin F_{n}$ and  we have that $\sum_{n=1}^{\infty}-\mu(\cup_{h\in F_{n}}D_{n}^{h})\log\frac{\mu(\cup_{h\in F_{n}}D_{n}^{h})}{|F_{n}|}<\infty$. 
Since $\{D_{n}^{h}\}_{n\in\N,h\in H}$ is a refinement of $\mathscr{P}_{\alpha,g}$ and entropy increases under refinements, we get the following inequality: 
\begin{align*}
\sum_{h\in H}-\mu(P_{g,h})\log \mu(P_{g,h})\leq \sum_{h\in H,n\in\N}-\mu(D^{h}_{n})\log \mu(D^{h}_{n}).
\end{align*}
 Note that the equipartition of a set maximizes entropy relative to any other partition of the same size. So, we can conclude  that 
\begin{align*}
\sum_{h\in H, n\in\N}-\mu(D^{h}_{n})\log \mu(D^{h}_{n})&\leq \sum_{n\in\N}\sum_{h\in F_{n}}-\mu(D^{h}_{n})\log \mu(D^{h}_{n})\\
&\leq  \sum_{n\in\N} -|F_{n}| \frac{\mu(\cup_{h\in F_{n}}D^{h}_{n})}{|F_{n}|}\log\frac{\mu(\cup_{h\in F_{n}}D^{h}_{n})}{|F_{n}|}\\
&\leq  \sum_{n\in\N} - \mu(\cup_{h\in F_{n}}D^{h}_{n})\log\frac{\mu(\cup_{h\in F_{n}}D^{h}_{n})}{|F_{n}|}
\end{align*}

So, working backwards, we can see that $\sum_{h\in H}-\mu(P_{g,h})\log \mu(P_{g,h})<\infty$.
\end{proof}
\end{section}
\begin{section}{Column Permutations of Towers}
Let $(X,\mu)$ be a standard atomless probability space and $\Z\curvearrowright^{S} (X,\mu)$ be an ergodic (free) p.m.p action.
\begin{definition}
    We call a pair $(T,B)$ consisting of a finite set $T\subset\Z$ and a measurable set $B\subset X$ a \textbf{tower} of $S$ if the sets in the collection $\{S^{n}B\}_{n\in T}$ are pairwise disjoint. Moreover, if $(T,B)$ is a tower such that $\mu\bigg(X\setminus(S^{T}B)\bigg)>1-\epsilon$ for some $\epsilon>0$, we say that $(T,B)$ is an $\epsilon$-\textbf{tower}. Given an $\epsilon$-tower $(T,B)$ and a partition $\mathscr{A}$ of $B$, we call the collection $\{S^{n}A\}_{n\in T}$ the \textbf{column} over $A\in\mathscr{A}$ in $(T,B)$ and the set $S^{n}A$ for some $n\in T$ a \textbf{column level} in $\{S^{n}A\}_{n\in T}$.  
\end{definition}
\begin{definition} Let $\epsilon>0$ be given and let $(T,B)$ be an $\epsilon$-tower of $S$. Also, assume we are given an ordered partition $\mathscr{P}=\{P_{1},\dots,P_{n}\}$ of $X$. We define the \textbf{pure partition} $\mathscr{A}_{\mathscr{P}}$ of $B$ relative to $\mathscr{P}$ to be the coarsest partition of $B$ such that, for every $A\in \mathscr{A}_{\mathscr{P}}$, each column level in the column over $A$ is contained in some element of $\mathscr{P}$. We also define the $(\mathscr{P},T)$\textbf{-name} of an atom $A\in \mathscr{A}_{\mathscr{P}}$ to be the indexed set $\{i_{t}\}_{t\in T}\subset \{1,\dots,n\}$ such that $S^{t}A\in P_{i_{t}}$.
\end{definition}
\begin{definition}
Let $(T,B)$ be an $\epsilon$-tower of $S$. Let $G$ be a countable discrete group and let $F\subset G$ be such that $|F|=|T|$.  Suppose we are given some finite set $I$ and a bijection $\phi_{i}:T\to F$ for each $i\in I$. Let $\mathscr{A}$ be a partition of $B$ such that $|\mathscr{A}|>|I|$ and let $\sigma:\mathscr{A}\to I$ be a surjection. We define a map $\alpha: T\times B\to F$ in the following way: given $A\in \mathscr{A}$, we set $\alpha(n,x)=\phi_{\sigma(A)}(n)$ for each $x\in A$. We call the map $\alpha$ a \textbf{column-to-tile map} of the tower $(T,B)$ to $F$. We say $\alpha$ is \textbf{pure} with respect to a partition $\mathscr{P}$ of $X$ if $\mathscr{A}$ is a refinement of the pure partition of $B$ relative to $\mathscr{P}$.
\end{definition}

\end{section}

\begin{section}{Actions of $\Z^{2}$ Induced by a Sequence of Column-to-Tile Maps}
Let $(X,\mu)$ be a standard probability space and $\Z\curvearrowright^{S} (X,\mu)$ be an ergodic p.m.p action. We will show how to create an action $\Z^{2}\curvearrowright ^{R}(X,\mu)$ from a sequence of column-to-tile maps of towers of $S$ which satisfy a few requirements. Note that any reference to a generator of $\Z^{2}$ will refer to either element from the pair of conventional generators $(1,0)$ and $(0,1)$ of $\Z^{2}$. \par

 First, suppose we are given a sequence $\{\epsilon_{n}\}_{n\in \N}$ such that $\epsilon_{n}<2^{-n}$ for all $n\in\N$. Assume for every $n\in\N$ we have an $\epsilon_{n}$-tower $(T_{n},B_{n})$ of $S$ such that $\frac{|T_{n}|}{|T_{n+1}|}<\epsilon_{n+1}$ and $T_{n}$ is an interval of nonnegative integers which contains $0$. We will also assume that there is some $N\in\N$ such that $|T_{n}|<2^{n^{3}}$ for all $n>N$. In addition, assume that there is a sequence $\{F_{n}\}_{n\in\N}$ of squares centered at the origin in $\Z^{2}$ such that $|F_{n}|=|T_{n}|$ and $|(F_{n}+F_{n+1})\setminus F_{n+1}|<\epsilon_{n+1}|F_{n+1}|$ for every $n\in\N$. Now, for each $n\in\N$, let $\alpha_{n}$ be a column-to-tile map of $(T_{n},B_{n})$ to $F_{n}\subset \Z^{2}$ with associated partition $\mathscr{A}_{n}$ of $B_{n}$.

   \par

Now, assume our column-to-tile maps satisfy the following properties. For every $n\in\N$, assume there is a finite set $C_{n}\subset \Z^{2}$ such that the tiles $\{F_{n}+c\}_{c\in C_{n}}$ are disjoint and $\bigsqcup_{c\in C_{n}}(F_{n}+c)= F_{n+1}$. We define, for every $n>1$, a function $\ell_{n}: B_{n}\to \{0,1\}^{T_{n}}$ given by 
\begin{align*}
\ell_{n}(b)(m)= \begin{cases}
	1 & \text{if } S^{m}x\in B_{n-1}\\
	0 & \text{if } S^{m}x\notin B_{n-1}
\end{cases}
\end{align*}
for $b\in B_{n}$ and $m\in T_{n}$. We make the assumption that if $A\in\mathscr{A}_{n}$, we have, for every $b,b'\in A$, that $\ell_{n}(b)=\ell_{n}(b')$.
Also, for every $n\in\N$ and a given $A\in\mathscr{A}_{n+1}$, we let $\bar{T}_{n+1}\subset T_{n+1}$ be the largest subset such that $S^{\bar{T}_{n+1}}A\in B_{n}$ and $S^{T_{n}+\bar{T}_{n+1}}A\subset S^{T_{n+1}}B_{n+1}$. So assume that, given $b\in A$, we have a distinct $c\in C_{n}$ for each $i\in \bar{T}_{n+1}$ so that $\alpha_{n+1}(i+j,b)=\alpha_{n}(j,S^{i}b)+c$ for all $j\in T_{n}$. \par

\begin{proposition} If for some $x\in X$ and $i,k\in\N$, $i<k$, we have that $x\in S^{T_{j}}B_{j}$, so that $x=S^{m_{j}}b_{j}$ for some $m_{j}\in T_{j}$ and $b_{j}\in B_{j}$, and there exists a $y\in B_{j}\cap S^{T_{j+1}}B_{j+1}$ such that $x\in S^{T_{j}}y\subset S^{T_{j+1}}B_{j+1}$ whenever $i\leq j\leq k$, then we can write $F_{k}$ as a disjoint union of translates of $F_{i}$ so that there is some $c\in\Z^{2}$ such that $\alpha_{k}(m_{k},b_{k})= \alpha_{i}(m_{i},b_{i})+c$.
\end{proposition}
\begin{proof}
Let $x\in X$ and $i,k\in\N$  be given such that $x\in S^{T_{j}}B_{j}$, so that $x=S^{m_{j}}b_{j}$ for some $m_{j}\in T_{j}$ and $b_{j}\in B_{j}$, and there exists a $y_{j}\in B_{j}\cap S^{T_{j+1}}B_{j+1}$ such that $x\in S^{T_{j}}\{y_{j}\}\subset S^{T_{j+1}}B_{j+1}$ whenever $i\leq j\leq k$. So, for every $j\in\{i,\dots,k\}$ there is a $n_{j+1}\in T_{j+1}$ such that $b_{j}=S^{n_{j+1}}b_{j+1}$. Thus, for every $j\in\{i,\dots,k\}$ there is a $c_{j}\in \Z$ so that, by construction, we have that
\begin{align*}
\alpha_{j+1}(m_{j+1},b_{j+1})&=\alpha_{j+1}(m_{j}+n_{j+1},b_{j+1})= \alpha_{j}(m_{j},S^{n_{j+1}}b_{j+1})+c_{j}\\
&=\alpha_{j}(t_{j},b_{j})+ c_{j}.
\end{align*}
If we set $c=\sum_{j=i}^{k-1} c_{j}$, we can see that recursively applying the above formula yields  $\alpha_{k}(m_{k},b_{k})= \alpha_{i}(m_{i},b_{i})+c$.
\end{proof}
 \begin{proposition}
Let $s\in \{(1,0),(0,1)\}$, one of the generators of $\Z^{2}$. For a.e. $x\in X$, there is an $N\in \N$ such that, for all $n\geq N$, the following conditions hold: \begin{enumerate}
\item $x\in S^{T_{n}}B_{n}$,

\item there exists a $y\in B_{n}\cap S^{T_{n+1}}B_{n+1}$ such that $x\in S^{T_{n}}\{y\}\subset S^{T_{n+1}}B_{n+1}$, and,
\item when $x$ is expressed as $x=S^{m_{1}}b$ where $m_{1}\in T_{n}$ and $b\in B_{n}$, we have a $m_{2}\in T_{n}$ such that $\alpha_{n}(m_{2},b)-\alpha_{n}(m_{1},b)=s$.
\end{enumerate}
 \end{proposition}
 \begin{proof}
 Let $s\in\Z^{2}$ be a given generator, and let $\mathscr{A}_{k}$ be the partition given above for every $k\in\N$. Let $S_{k}$ be the union of the top $|T_{k-1}|-1$ and bottom $|T_{k-1}|-1$ tower levels of $(T_{k},B_{k})$. Set \begin{align*}
 E_{k}=(X\setminus S^{T_{k}}B_{k})\cup (X\setminus S^{T_{k-1}}B_{k-1})\cup \{S^{t}b\in S^{T_{k}}B_{k}\mid s+\alpha_{k}(t,b)\notin F_{k}\}\cup S_{k}.
\end{align*} 

The proof follows if we can show that $\mu(\bigcap_{n=1}^{\infty}\bigcup_{k=n}^{\infty}E_{k})=0$. This can be accomplished by showing that $\sum \mu(E_{k})<\infty$ and then applying the Borel-Cantelli Lemma. First, since $S^{T_{n}}B_{n}$ is an $\epsilon_{n}$-tower for every $n\in\N$, we have that $\mu(X\setminus S^{T_{k}}B_{k})<\epsilon_{k}$. Similarly, we have $\mu(X\setminus S^{T_{k-1}}B_{k-1})<\epsilon_{k-1}$. Also, by construction, we have that $\mu(S_{k})<\frac{2|T_{k-1}|}{|T_{k}|}<2\epsilon_{k}$ and $\mu(\{S^{t}b\in S^{T_{k}}B_{k}\mid s+\alpha_{k}(t,b)\notin F_{k}\})<\epsilon_{k}$. So, we have that $\mu(E_{k})< 4\epsilon_{k}+\epsilon_{k-1}$ so that $\sum \mu(E_{k})<\infty$ since we chose $\{\epsilon_{k}\}_{k\in\N}$ to be summable.
 \end{proof}
 
Let $s\in\Z^{2}$ be a standard generator. For all $n\in\N$, we can define a sequence of partial maps $\sigma_{s,n}: S^{T_{n}}B_{n}\to S^{T_{n}}B_{n}$ in the following way: given some $t\in T_{n}$ and $b\in B_{n}$ such that there is some $t'\in T_{n}$ such that $\alpha_{n}(t',b)-\alpha_{n}(t,b)= s$, set $\sigma_{s,n}(S^{t}b)=S^{t'}b$.
 
\begin{proposition}
 Let $s\in\Z^{2}$ be a standard generator and $x\in X$. There is an $N\in\N$ such that, for all $n_{1},n_{2}>N$, we have that $x\in S^{T_{n_{1}}}B_{n_{1}}\cap S^{T_{n_{2}}}B_{n_{2}}$ and $\sigma_{s,n_{1}}(x)=\sigma_{s,n_{2}}(x)$.
\end{proposition}

\begin{proof} Let a standard generator $s\in\Z^{2}$ and $x\in X$ be given, and let $N\in\N$ be given so that the property from Proposition 4.2 holds. Let $n_{2}>n_{1}>N$ be given. So, we have $t_{1}\in T_{n_{1}}$, $b_{1}\in B_{n_{1}}$, $t_{2}\in T_{n_{2}}$ ,and $b_{2}\in B_{n_{2}}$ such that $x=S^{t_{1}}b_{1}=S^{t_{2}}b_{2}$, and we have $t_{1}'\in T_{n_{1}}$ and $t_{2}'\in T_{n_{2}}$ such that $\alpha_{n_{1}}(t_{1}',b_{1})=s+\alpha_{n_{1}}(t_{1},b_{1})$ and $\alpha_{n_{2}}(t_{2}',b_{2})=s+\alpha_{n_{2}}(t_{2},b_{2})$. We must show that $S^{t_{1}'}b_{1}=S^{t_{2}'}b_{2}$. Since $S^{t_{1}}b_{1}=x\in S^{T_{n_{2}}}B_{n_{2}}$, we can find a $t_{3}\in T_{n_{2}}$ such that $S^{t_{1}'}b_{1}=S^{t_{3}}b_{2}$. By Proposition 4.1, we have that $F_{n_{2}}$ can be written as a disjoint union of translates of $F_{n_{1}}$ so that there is some $c\in\Z^{2}$ such that \begin{equation}\alpha_{n_{2}}(t_{3},b_{2})=\alpha_{n_{1}}(t_{1}',b_{1})+c
\end{equation}
and 
\begin{equation}
\alpha_{n_{2}}(t_{2},b_{2})=\alpha_{n_{1}}(t_{1},b_{1})+c
\end{equation}
so that
\begin{align*}
\alpha_{n_{2}}(t_{2}',b_{2})&=\alpha_{n_{1}}(t_{1}',b_{1})+\alpha_{n_{2}}(t_{2},b_{2})-\alpha_{n_{1}}(t_{1},b_{1})\\
&\stackrel{(2)}{=}\alpha_{n_{1}}(t_{1}',b_{1})+c\\
&\stackrel{(1)}{=}\alpha_{n_{2}}(t_{3},b_{2}).
\end{align*}
Since $\alpha_{n_{2}}(\cdot,b_{2})$ is a bijection, we have that $t_{2}'=t_{3}$ so that $S^{t_{2}'}b_{2}=S^{t_{3}}b_{2}=S^{t_{1}'}b_{1}$.
\end{proof}
Now, we are ready to define the action $\Z^{2}\curvearrowright X$. Let $s\in\Z^{2}$ be a standard generator and $x\in X$. Thus, we can find an $N\in\N$ such that, for all $n,m>N$, we have that $x\in S^{T_{m}}B_{m}\cap S^{T_{n}}B_{n}$ and $\sigma_{s,m}(x)=\sigma_{s,n}(x)$. So, the action will be defined by setting $sx=\sigma_{s,n}(x)$ for some $n>N$.

\begin{proposition}
The action of $\Z^{2}\curvearrowright^{R} X$ defined above is orbit equivalent to $\Z\curvearrowright^{S} X$. 
\end{proposition}
\begin{proof}
Let $x\in X$ be given such that there is an $N\in\N$ such that for all $m,n>N$ we have that $\sigma_{s,m}(x)=\sigma_{s,n}(x)$ for any standard generator $s\in\Z^{2} $. So, for all $m,n>N$, there is a $t_{m}\in T_{m}$, a $t_{n}\in T_{n}$, a $b_{m}\in B_{m}$ and a $b_{n}\in B_{n}$ such that $x=S^{t_{m}}b_{m}=S^{t_{n}}b_{n}$ and  $g,h\in\Z^{2}$ such that $g=\alpha_{n}(1+t_{n},b_{n})-\alpha_{n}(t_{n},b_{n})= \alpha_{m}(1+t_{m},b_{m})-\alpha_{m}(t_{m},b_{m})$ and $h=\alpha_{n}(-1+t_{n},b_{n})-\alpha_{n}(t_{n},b_{n})= \alpha_{m}(-1+t_{m},b_{m})-\alpha_{m}(t_{m},b_{m})$. So, $R^{g}x=S^{1}x$ and $R^{h}x=S^{-1}x$. This means that for all $\ell\in\Z$, we have that there is a $v\in \Z^{2}$ such that $R^{v}x=S^{\ell}x$. Now, suppose for some $w\in \Z^{2}$ that $R^{w}x=S^{\ell}x$ also. So, for some $N\in\N$, we have that, for all $m,n>N$, there is a $t_{m}\in T_{m}$, a $t_{n}\in T_{n}$, a $b_{m}\in B_{m}$ and a $b_{n}\in B_{n}$ such that $x=S^{t_{m}}b_{m}=S^{t_{n}}b_{n}$ such that $v=\alpha_{n}(1+t_{n},b_{n})-\alpha_{n}(t_{n},b_{n})= \alpha_{m}(1+t_{m},b_{m})-\alpha_{m}(t_{m},b_{m})=w$. \par
Similarly, for a standard generator $s\in\Z^{2}$, we can find an $N\in\N$ where, for every $n>N$ we have a $v\in\Z$ so that, whenever $x=S^{t}b\in S^{T_{n}}B_{n}$, we have a $t'\in T_{n}$ such that $t-t'=v$ and $\alpha_{n}(t',b)-\alpha_{n}(t,b) =s$. Thus, $R^{s}x=S^{v}x$. The fact that we bijectively map onto the orbit of $x$ follows similarly to the first case.
\end{proof}
Let $\lambda: X\times \Z\to \Z^{2} $ and $\kappa: X\times \Z^{2}\to \Z$ be the cocycles induced by the orbit equivalence between $\Z\curvearrowright^{S} (X,\mu)$ and $\Z^{2}\curvearrowright^{R} (X,\mu)$. Let $x\in X$. We can express $\lambda$ in the following way: given an $n\in\Z$, we can find an $N\in\N$ such that for all $k>N$ we have that $x\in S^{T_{k}}B_{k}$ and, when $x$ is expressed as $x=S^{t}b$ where $t\in T_{k}$ and $b\in B_{k}$, we have that $n+t\in T_{k}$ and $\lambda(x,n)=\alpha_{k}(n+t,b)-\alpha_{k}(t,b)$. For $\kappa$, given a $g\in\Z^{2}$, we can find an $N\in\N$ such that $x\in S^{T_{N}}B_{N}$ and, when $x$ is expressed as $x=S^{t}b$ where $t\in T_{N}$ and $b\in B_{N}$, we have that there is a $t'\in T_{N}$ such that $\alpha_{N}(t',b)-\alpha_{N}(t,b)=g$, so we get that $\kappa(x,g)=t'-t$.\par It will be useful for us later to have an estimate on the measure of the set of points whose $\lambda$ cocycle value for some element in $\Z$ is fixed at a given stage in the construction. By saying that a point $x\in X$ has its $\lambda$ cocycle value fixed for $k\in\Z$ at stage $N$, we mean that $N$ is the smallest natural number such that, for all $n\geq N$, there exists a $t\in T_{n}$ and a $b\in B_{n}$ such that $x=S^{t}b$ and $\lambda(x,k)=\alpha_{n}(k+t,b)-\alpha_{n}(t,b)$.
\begin{lemma} Let $k\in\Z$ be given. Let $D^{k}_{n}$ be the set of points whose $\lambda$ cocycle value for $k$ is fixed at stage $n$. Then, for $n\geq 2$, we have that
\begin{align*}
\mu(D_{n}^{k})< \epsilon_{n-1}+3\sum_{i=n}^{\infty}\epsilon_{i}+k\sum_{i=n}^{\infty}\frac{1}{|T_{i-1}|}
\end{align*}
\end{lemma}
\begin{proof}
We only need to show that the following bounds for $\mu(D_{n}^{k})$ for all $n\geq 2$:
\begin{align*}
\mu(S^{T_{n}}B_{n})-\sum_{i=1}^{n-1}\mu(D_{i}^{k})-3\sum_{i=n}^{\infty}\epsilon_{i+1}-k\sum_{i=n}^{\infty}\frac{1}{|T_{i}|}\leq \mu(D_{n}^{k})\leq\mu(S^{T_{n}}B_{n})-\sum_{i=1}^{n-1}\mu(D_{i}^{k})
\end{align*}
Indeed, assuming these hold we have
\begin{align*}
\mu(D_{n}^{k})&\leq\mu(S^{T_{n}}B_{n})-\sum_{i=1}^{n-1}\mu(D_{i}^{k})\\
&\leq \mu(S^{T_{n}}B_{n})-\bigg(\mu(S^{T_{n-1}}B_{n-1})-\sum_{i=1}^{n-2}\mu(D_{i}^{k})-3\sum_{i=n-1}^{\infty}\epsilon_{i+1}-k\sum_{i=n-1}^{\infty}\frac{1}{|T_{i}|}\bigg)-\sum_{i=1}^{n-2}\mu(D_{i}^{k})\\
&\leq  \mu(S^{T_{n}}B_{n})-\mu(S^{T_{n-1}}B_{n-1})+3\sum_{i=n-1}^{\infty}\epsilon_{i+1}+k\sum_{i=n-1}^{\infty}\frac{1}{|T_{i}|}\\
&<1-(1-\epsilon_{n-1})+3\sum_{i=n-1}^{\infty}\epsilon_{i+1}+k\sum_{i=n-1}^{\infty}\frac{1}{|T_{i}|}\\
&=\epsilon_{n-1}+3\sum_{i=n}^{\infty}\epsilon_{i}+k\sum_{i=n}^{\infty}\frac{1}{|T_{i-1}|}.
\end{align*}
First, the upper bound simply follows from the fact that $\bigsqcup_{i=1}^{n}D^{k}_{i}\subset S^{T_{n}}B_{n} $.
For the lower bound, note that a point $x\in S^{T_{n}}B_{n}$, which we represent as $x=S^{t}b$ where $t\in T_{n}$ and $b\in B_{n}$, is an element of  $D_{n}^{k}$ if and only if the following conditions hold:
\begin{enumerate}
\item $x\in S^{T_{i}}B_{i}$ for every $i\geq n$
\item $t\in T_{i}\cap (-k+T_{i})$ for every $i\geq n$
\item $x$ is not in the top or bottom $|T_{i}|$ levels of the tower $(T_{i+1},B_{i+1})$ for every $i\geq n$.
\item $x\notin \bigcup_{i=1}^{n-1}D^{k}_{i}$
\end{enumerate}
for every $n\in\N$.
Thus, for $n>1$,
\begin{align*}
\mu(D_{n}^{k})&> \mu(S^{T_{n}}B_{n})-\sum_{i-1}^{n-1}\mu(D^{k}_{i})-\sum_{i=n}^{\infty}\mu(X\setminus S^{T_{i+1}}B_{i+1})\\
&\text{ }\text{ }\text{ }-\sum_{i=n}^{\infty}k \frac{\mu(S^{T_{n}}B_{n})}{|T_{n}|}-\sum_{i=n}^{\infty} \frac{2|T_{n}|\mu(S^{T_{n+1}}B_{n+1})}{|T_{n+1}|}\\
&>\mu(S^{T_{n}}B_{n})-\sum_{i-1}^{n-1}\mu(D^{k}_{i})-\sum_{i=n}^{\infty}\epsilon_{n+1}-k\sum_{i=n}^{\infty}\frac{1}{|T_{n}|}-\sum_{n=1}^{\infty}2\epsilon_{n+1}\\
&=\mu(S^{T_{n}}B_{n})-\sum_{i-1}^{n-1}\mu(D^{k}_{i})-3\sum_{i=n}^{\infty}\epsilon_{n+1}-k\sum_{i=n}^{\infty}\frac{1}{|T_{n}|}.
\end{align*}

\end{proof}
Now, we will show a similar result for the set of points whose $\kappa$ cocycle value for a given element of $\Z^{2}$ is fixed at a given stage $n$. 
\begin{lemma} Let $g\in\Z^{2}$ be given. Let $E^{g}_{n}$ be the set of points whose $\kappa$ cocycle value for $g$ is fixed at stage $n$. Then, for $n\geq 2$, we have that
\begin{align*}
\mu(E_{n}^{g})< \epsilon_{n-1}+3\sum_{i=n}^{\infty}\epsilon_{i}+\sum_{i=n}^{\infty}\frac{|F_{i-1}\setminus g+F_{i-1}|}{|F_{i-1}|}
\end{align*}
\end{lemma}
\begin{proof}
Similar to the proof of Lemma 4.5, we need to show that 
\begin{align*}
\mu(S^{T_{n}}B_{n})-\sum_{i=1}^{n-1}\mu(E_{g}^{k})-3\sum_{i=n}^{\infty}\epsilon_{i+1}-\sum_{i=n}^{\infty}\frac{|F_{i}\setminus (g+F_{i})|}{|F_{i}|}\leq \mu(E_{g}^{k})\leq\mu(S^{T_{n}}B_{n})-\sum_{i=1}^{n-1}\mu(E_{g}^{k})
\end{align*}
since
\begin{align*}
\mu(E_{n}^{g})&\leq \mu(S^{T_{n}}B_{n})-\sum_{i=1}^{n-1}\mu(E_{n}^{g})\\
&\leq \mu(S^{T_{n}}B_{n})-\bigg(\mu(S^{T_{n-1}}B_{n-1})-\sum_{i=1}^{n-2}\mu(E_{n}^{g})-3\sum_{i=n-1}^{\infty}\epsilon_{i+1}-\sum_{i=n}^{\infty}\frac{|F_{i}\setminus g+F_{i}|}{|F_{i}|}\bigg)-\sum_{i=1}^{n-2}\mu(E_{i}^{g})\\
&\leq  \mu(S^{T_{n}}B_{n})-\mu(S^{T_{n-1}}B_{n-1})+3\sum_{i=n-1}^{\infty}\epsilon_{i+1}+\sum_{i=n}^{\infty}\frac{|F_{i}\setminus g+F_{i}|}{|F_{i}|}\\
&<1-(1-\epsilon_{n-1})+3\sum_{i=n-1}^{\infty}\epsilon_{i+1}+\sum_{i=n}^{\infty}\frac{|F_{i}\setminus g+F_{i}|}{|F_{i}|}\\
&=\epsilon_{n-1}+3\sum_{i=n}^{\infty}\epsilon_{i}+\sum_{i=n}^{\infty}\frac{|F_{i}\setminus g+F_{i}|}{|F_{i}|}.
\end{align*}
The upper bound simply follows from the fact that $\bigsqcup_{i=1}^{n}E^{g}_{i}\subset S^{T_{n}}B_{n} $.
For the lower bound, note that a point $x\in S^{T_{n}}B_{n}$, representing $x$ as $x=S^{t}b$ where $t\in T_{n}$ and $b\in B_{n}$, is also an element of $E_{n}^{g}$ if and only if
\begin{enumerate}
\item $x\in S^{T_{i}}B_{i}$ for every $i\geq n$
\item $\alpha_{i}(t,b)\in F_{i}\cap (-g+F_{i})$ for every $i\geq n$
\item $x$ is not in the top or bottom $|T_{i}|$ levels of the tower $S^{T_{i+1}}B_{i+1}$ for every $i\geq n$.
\item $x\notin \bigcup_{i=1}^{n-1}E^{g}_{i}$
\end{enumerate}
for every $n\in\N$.
Thus, for $n>1$,
\begin{align*}
\mu(E_{n}^{g})&> \mu(S^{T_{n}}B_{n})-\sum_{i-1}^{n-1}\mu(E^{g}_{i})-\sum_{i=n}^{\infty}\mu(X\setminus S^{T_{i+1}}B_{i+1})\\
&\text{ }\text{ }\text{ }-\sum_{i=n}^{\infty}\frac{|F_{i}\setminus g+F_{i}|}{|F_{i}|}-\sum_{i=n}^{\infty} \frac{2|T_{n}|\mu(S^{T_{n+1}}B_{n+1})}{|T_{n+1}|}\\
&>\mu(S^{T_{n}}B_{n})-\sum_{i-1}^{n-1}\mu(E^{g}_{i})-\sum_{i-n}^{\infty}\epsilon_{n+1}-\sum_{i=n}^{\infty}\frac{|F_{i}\setminus g+F_{i}|}{|F_{i}|}-\sum_{n=1}^{\infty}2\epsilon_{n+1}\\
&=\mu(S^{T_{n}}B_{n})-\sum_{i-1}^{n-1}\mu(E^{g}_{i})-3\sum_{i=n}^{\infty}\epsilon_{n+1}-\sum_{i=n}^{\infty}\frac{|F_{i}\setminus g+F_{i}|}{|F_{i}|}.
\end{align*}
 
\end{proof}
\begin{proposition}
The orbit equivalence between $\Z\curvearrowright X$ and $\Z^{2}\curvearrowright X$ is Shannon.
\end{proposition}
\begin{proof}
Let $k\in\Z$ be given. First, we will show that the cocycle partition of $\lambda:X\times\Z\to \Z^{2}$ for $k$ has finite Shannon entropy. For a given $g\in\Z^{2}$, let $P_{g}=\{x\in X\mid \lambda(x,k)=g\}$ and let $\mathscr{P}=\{P_{g}\mid g\in\Z^{2}\}$. For every $n\in\N$ and $g\in\Z^{2}$, set $D^{k,g}_{n}=D^{k}_{n}\cap P_{g}$ where $D^{k}_{n}$ is the set of points whose $\lambda$ cocycle value for $k$ is fixed at stage $n$. So, for every $g\in \Z^{2}$, $\{D^{k,g}_{n}\}_{n\in\N}$ is a partition of $P_{g}$. Recall that, for every $n\in\N$, $F_{n}\subset \Z^{2}$ is the set that corresponds to the column-to-tile map $\alpha_{n}$ and
let $\bar{F}_{n}=F_{n}-F_{n}$. Note that $D^{k,g}_{n}=\emptyset$ if $g\notin \bar{F}_{n}$. So, in order to apply Proposition 2.2, the only property left for us to show is that  $\sum_{n=1}^{\infty}\mu(D_{n}^{k})\log\frac{\mu(D_{n}^{k})}{|\bar{F}_{n}|}<\infty$ since $D_{n}^{k}=\cup_{g\in \Z^{2}}D_{n}^{k,g}$.\par
By Lemma 4.5, we have that $\mu(D_{n}^{k})< \epsilon_{n-1}+3\sum_{i=n}^{\infty}\epsilon_{i}+k\sum_{i=n}^{\infty}\frac{1}{|T_{i-1}|}$. By choosing $\epsilon_{n}<2^{-n}$ in our construction, which also implies that $2^{i}<|T_{i}|$ for all $i>1$, we have that $\mu(D_{n}^{k})< 2^{-(n-1)}+3(2^{-n+1})+k(2^{-n+2})$ for all $n$. Thus, there is an $N\in\N$ such that for all $n>N$, we have $\mu(D_{n}^{k})<2^{-n/2}$. By construction, we have $|F_{n}|<2^{n^{2}}$ for $n>N$ and $|\bar{F}_{n}|=4|F_{n}|$. So,  this implies that 
\begin{align*}
\sum_{n>N} - \mu(D^{k}_{n})\log\frac{\mu(D^{k}_{n})}{|\bar{F}_{n}|}&\leq \sum_{n>N} - 2^{-n/2}\log\frac{2^{-n/2}}{|\bar{F}_{n}|}\\
&\leq \sum_{n>N} - 2^{-n/2}\log\frac{2^{-n/2}}{4|F_{n}|}\\
&\leq \sum_{n>N}-2^{-n/2}\log 2^{-n^{2}-n-2}\\
&<\infty
\end{align*}
So, we can apply Proposition 2.2 to conclude that $\mathscr{P}$ has finite Shannon entropy.
\par

Now, we will show that the Shannon entropy of each partition associated to $\kappa: X\times \Z^{2}\to \Z$ is finite. Fix $g\in \Z^{2}$. Let $P_{k}=\{x\in X\mid \kappa(x,g)=k\}$ for every $k\in \Z$. For every $n\in\N$ and $k\in\Z$, set $E^{g,k}_{n}=E^{g}_{n}\cap P_{k}$. So, for every $k\in\Z$, we have that $\{E^{g,k}_{n}\}_{n\in\N}$ is a partition of $P_{k}$. 
Just like the $\lambda$ cocycle case, we have that $E^{g,k}_{n}=\emptyset$ if $k\notin \bar{T}_{n}$ where $\bar{T}_{n}=T_{n}-T_{n}$. 
Lemma 4.6 gives us that 
\begin{align*}
\mu(\bigcup_{k\in \bar{T}_n}E_{n}^{g,k})=\mu(E_{n}^{g})< \epsilon_{n-1}+3\sum_{i=n}^{\infty}\epsilon_{i}+\sum_{i=n}^{\infty}\frac{|F_{i-1}\setminus (g+F_{i-1})|}{|F_{i-1}|}.
\end{align*}
Note that we can find an $N\in\N$ such that for all $n>N$ we have that $\sum_{i=n}^{\infty}\frac{|F_{i-1}\setminus (g+F_{i-1})|}{|F_{i-1}|}<2^{-n/4}$ since $2^{i-1}<|F_{i-1}|$ and $\frac{|F_{i-1}\setminus (g+F_{i-1})|}{|F_{i-1}|}$ is approximately $\frac{1}{\sqrt{|F_{i-1}|}}$. Thus, there is an $N\in\N$ such that, for all $n>N$, $\sum_{i=n}^{\infty}\frac{|F_{i-1}\setminus (g+F_{i-1})|}{|F_{i-1}|}<2^{-n/4}$ and $\epsilon_{n}<\frac{1}{2^{n}}$ so that we have $\mu(E^{g}_{n})<2^{-n/8}$ for all $n>N$. 
By taking $|T_{n}|<2^{n^{2}}$ for $n>N$ and noting that $|\bar{T}_{n}|=2|T_{n}|$, we have that 
\begin{align*}
\sum_{n>N} -|\bar{T}_{n}| \frac{\mu(E^{g}_{n})}{|\bar{T}_{n}|}\log\frac{\mu(E^{g}_{n})}{|\bar{T}_{n}|}&\leq \sum_{n>N} -2^{-n/8}\log\frac{2^{-n/8}}{|\bar{T}_{n}|}\\
&\leq \sum_{n>N} -2^{-n/8}\log\frac{2^{-n/8}}{2|T_{n}|}\\
&\leq \sum_{n>N}-2^{-n/8}\log 2^{-n^{2}-\frac{n}{8}-1}\\
&<\infty
\end{align*}
Since we have that
\begin{align*}
\sum_{n>N} -
\mu(\bigcup_{k\in\bar{T}_{n}}E^{g,k}_{n})\log\frac{\mu(\bigcup_{k\in\bar{T}_{n}}E^{g,k}_{n})}{|\bar{T}_{n}|}=\sum_{n>N} -|\bar{T}_{n}| \frac{\mu(E^{g}_{n})}{|\bar{T}_{n}|}\log\frac{\mu(E^{g}_{n})}{|\bar{T}_{n}|},
\end{align*}
we can apply Proposition 2.2 to conclude that the Shannon entropy of each partition associated to $\kappa$ is finite.
\end{proof}

\end{section}
\begin{section}{Summary of the Construction from Fieldsteel and Friedman}
    Now, we will summarize the construction involved in Fieldsteel-Friedman. Later, we will show how this construction can be modified to produce a theorem similar to Theorem 1.16 for Shannon orbit equivalence.
\begin{remark}
If $L\in\N$, we will use the following notation to represent a $d$-cube in $\N^{d}$
\begin{align*}
C_{L}=\{v=(v_{1},\dots,v_{d})\in\Z^{d}\mid 1\leq v_{i}\leq L, i=1,2,\dots, d\}
\end{align*}
and the following notation to represent the $d$-cube centered at the origin of side length $2L+1$ in $\Z^{d}$
\begin{align*}
\bar{C}_{L}=\{v=(v_{1},\dots,v_{d})\in\Z^{d}\mid |v_{i}|\leq L, i=1,2,\dots, d\}
\end{align*}
\end{remark}
\begin{definition}
Let $J,K,L\in N$. We say a permutation $\pi$ of $C_{L}$ is a $(J,K,L)$-\textbf{permutation} if there exists a $u\in\Z^{d}$ and $v\in \bar{C}_{K}$ such that for all $w\in C_{j}+u$, we have that $\pi(w)=w+v$. The vector $v$ is called the \textbf{translation vector} of $\pi$ and will be denoted by $tv(\pi)$. We also call a permutation $\pi$ of $C_{L}+w$ for some $w\in\Z^{d}$ a $(J,K,L)$-\textbf{permutation} if it is a $(J,K,L)$-permutation of $C_{L}$ up to translation by $w$. Moreover, a bijection $\pi$ of $\Z^{d}$ is called a $(J,K,L)$-\textbf{blocked bijection} if there is a collection $\{C_{L}+w_{i}\}_{i\in I}$ of disjoint translates of $C_{L}$ for which $\pi$ acts on each $C_{L}+w_{i}$ by a $(J,K,L)$-permutation and acts as the identity on the rest of $\Z^{d}$. If $C_{J}+u_{i}+w_{i}$ is the subset of $C_{L}+w_{i}$ on which $\pi$ acts by translation, then we call the sets $C_{L}+w_{i}$ the \textbf{blocks} of $\pi$, the sets $C_{J}+u_{i}+w_{i}$ the \textbf{rigid blocks} of $\pi$, and the sets $C_{L}+w_{i}\setminus C_{J}+u_{i}+w_{i}$ \textbf{filler sets} of $\pi$.\par
Now, we say a cocycle $\alpha$ for a probability measure preserving $\Z^{d}$- action $S$ of the probability space $(X,\mu)$ is a $(J,K,L)$-\textbf{blocked cocycle} if, for almost every $x\in X$, there is a $(J,K,L)$-blocked bijection $\pi_{x}$ of $\Z^{d}$ such that $\alpha(x,\cdot)=\pi_{x}(\cdot)-\pi_{x}(0)$, and the blocks, rigid blocks, and the permutations on them are measurably selected. If $x\in X$ and $C_{L}+w_{i}$ is a block of $\pi_{x}$, we call the set $S^{C_{L}+w_{i}}(x)$ a \textbf{block} of $\alpha$. We make a similar definition for rigid blocks and filler sets. In addition to being a a $(J,K,L)$-blocked cocycle, we say $\alpha$ is a a $(J,K,L,\epsilon)$-\textbf{blocked cocycle} of $S$ if the measure of points in $X$ which are not in a block of $\alpha$ is less than $\epsilon$.
\end{definition}
\begin{remark}
Given a sequence of cocycles $\{\alpha_{i}\}_{i\in\N}$ of an action $\Z^{d}\curvearrowright (X,\mu)$, we say that $\{\alpha_{i}\}_{i\in\N}$ converges to a cocycle $\alpha$ if, for almost every $x\in X$ and every $v\in\Z^{d}$, we have that $\lim_{i\to\infty}\alpha_{i}(v,x)=\alpha(v,x)$.
\end{remark}
Lemma 1 in Fieldsteel-Friedman shows that, given a summable sequence $\{\epsilon_{i}\}_{i\in 
\N}$ in $(0,1)$ and an arbitrary sequence $\{K_{i}\}_{i\in \N}$ of positive integers, then, given a sequence of positive integers $\{J_{i}\}_{i\in \N}$ which grow at a sufficiently fast rate and a sequence of positive integers $\{L_{i}\}_{i\in \N}$ such that each $L_{i}$ is sufficiently close to $J_{i}$, we can find a sequence of cocycles $\{\alpha_{i}\}_{i\in\N}$ of an action $\Z^{d}\curvearrowright^{T}(X,\mu)$ which converges to a cocycle $\alpha$ such that the corresponding orbit equivalence between $T$ and $T^{\alpha^{-1}}$ is, for all $a\geq 1$ , a weak-$a$-equivalence and, for all $b\in(0,1)$, a strong-$b$-equivalence.

To get mixing, Fieldsteel and Friedman construct a sequence of blocked cocycles by choosing parameters so that the ergodic theorem and the Rokhlin lemma guarantee a relative mixing condition for each cocycle. In turn, this gives mixing once we pass to the limit, guaranteed by Lemma 1, of our sequence of blocked cocycles.
\end{section}
\begin{section}{Fieldsteel-Friedman Orbit Equivalence Is Shannon}
For the rest of this section, let $\Z\curvearrowright^{S} (X,\mu)$ be an ergodic probability measure preserving action. Also, let $\{K_{n}\}_{n\in\N}$ and $\{L_{n}\}_{n\in\N}$ be sequences of positive integers, and let $\{\epsilon_{n}\}_{n\in\N}$ be a sequence such that $\epsilon_{n}<2^{-n}$ for every $n\in N$. Assume we have a sequence of cocycles $\{\alpha_{n}\}_{n\in\N}$ of $S$ taking values in $\Z$ such that $\beta_{n}=\alpha_{n}\circ \alpha_{n-1}^{-1}$ is a $(K_{n}-L_{n}, K_{n},L_{n},\epsilon_{n})$-blocked cocycle. For every $n\in\N$, we set $S_{n}=S^{\alpha_{n}^{-1}}$ and  $B_{n}=2(\sum_{j=1}^{n}L_{j})$. Also, assume that for every $n\in\N$ we have that $\frac{K_{n}}{L_{n}}<\epsilon_{n}$, $\frac{n+B_{n}}{L_{n+1}-K_{n+1}}<\epsilon_{n+1}$, and, for large enough $n$, $K_{n}<2^{n^{3}}$.
\begin{proposition}
The sequence of cocycles $\{\alpha_{i}\}_{i\in\N}$ converges to a cocycle $\alpha$ given, for every $n\in \Z$ and almost every $x\in X$, by the formula $\lim_{i\to\infty}\alpha_{i}(x,n)=\alpha(x,n)$.
\end{proposition}
The proof of this proposition is an adaption of the proof of Lemma 1 from Fieldsteel and Friedman \cite{FF}.
\begin{proof}
First, set 
\begin{align*}
F_{i}=\{x\mid S_{i}^{\bar{C}_{i+B_{i}}}(x) \text{ is contained in a rigid block of } \beta_{i+1}\}
\end{align*}
We will first show that $\mu(\cup_{i=1}^{\infty}\cap_{j=i}^{\infty}F_{j})=1$, which we can accomplish by showing that $\sum_{j=1}^{\infty}\mu(X\setminus F_{j})<\infty$ and applying Borel-Cantelli. To show this, note that $x\notin F_{j}$ in three possible cases:
\begin{enumerate}
\item $x$ is not in a block of $\beta_{j+1}$,
\item $x$ is in a block of $\beta_{j+1}$, but not a rigid block, or
\item $x$ is in a rigid block of $\beta_{j+1}$, but not all of $S_{j}^{\bar{C}_{j+B_{j}}}(x)$.
\end{enumerate}
Note that the measure of the set of points that satisfy case 1 is $<\epsilon_{j+1}$, that satisfy case 2 is $\frac{K_{j+1}}{L_{j+1}}<\epsilon_{j+1}$, and that satisfy case 3 is $\frac{2(j+B_{j})}{L_{j+1}-K_{j+1}}<2\epsilon_{j+1}$. Thus, $\sum_{j=1}^{\infty}\mu(X\setminus F_{j})<\sum_{j=1}^{\infty}4\epsilon_{j+1}<\infty$ since for every $j\in \N$ we have that $\epsilon_{j}<2^{-j}$. Let $x\in \cup_{i=1}^{\infty}\cap_{j=i}^{\infty}F_{j}$ and $n\in\Z$. So, we can find an $i$ so that $|n|<i$ and $x\in \cap_{j=i}^{\infty}F_{j}$. For all $j\geq i$, we have
\begin{align*}
\alpha_{j}(x,n)=\alpha_{i}(x,n)=\lim_{k\to\infty}\alpha_{k}(x,n)=\alpha(x,n).
\end{align*}
Hence, $\alpha$ is a well defined function. \par 

We also have that $\alpha$ satisfies the cocycle identity since, given almost any $x\in X$ and given $m,n\in\Z$, we can find an $i\in\N$ such that $\alpha(x,n+m)=\alpha_{i}(x,n+m)$, $\alpha(x,m)=\alpha_{i}(x,m)$, and $\alpha(S^{m}x,n)=\alpha_{i}(S^{m}x,n)$ so that 
\begin{align*}
\alpha(x,n+m)=\alpha_{i}(x,n+m)= \alpha_{i}(x,m)+\alpha_{i}(S^{m}x,n)=\alpha(x,m)+\alpha(S^{m}x,n).
\end{align*}
A very similar argument shows that for almost every $x\in X$, we have that $\alpha(x,\cdot)$ is injective. Now, we just need to show that, for almost every $x\in X$, the function $\alpha(x,\cdot)$ is surjective. Let $x\in \cup_{i=1}^{\infty}\cap_{j=i}^{\infty}F_{j}$ and $n\in\Z$. Then, we can find an $i\in\N$ such that $x\in \cap_{j=i}^{\infty}F_{j}$. This implies that we can find a $j\in\N$ such that $j+B_{j}>|n|$. Since $\alpha_{j}(x,\cdot)$ is bijective, we can find an $m\in\Z$ such that $\alpha_{j}(x,m)=n$. Hence, we have that $\alpha(x,m)=n$ since $\alpha_{j}(x,m)=\alpha_{k}(x,m)$ for every $k\geq j$. 
\end{proof}
\begin{definition} Let $\alpha$ be a cocycle which is a limit of the cocycles $\{\alpha_{n}\}_{n\in\N}$ and fix a $k\in\Z$. We say a point $x\in X$ has its $\alpha$ \textbf{cocycle value for} $k$ \textbf{fixed at stage} $N$ if $N$ is the smallest natural number such that $\alpha_{N}(x,k)=\alpha(x,k)$. Recall that this implies that, for every $n\geq N$, $\alpha_{n}(x,k)=\alpha(x,k)$. We will let $D^{k}_{N}$ denote the set of all points whose cocycle value for $k$ is fixed at stage $N$.
\end{definition}
\begin{lemma} For the cocycle $\alpha$ for $S$ given above, we have for a given $k\in\Z$ and  $n>1$ that \begin{align*}
\mu(D^{k}_{n})<2\sum_{i=n-1}^{\infty}\epsilon_{i}+\sum_{i=n-1}^{\infty}\frac{|k|+B_{n-2}}{L_{i}}
\end{align*}
\end{lemma}
\begin{proof}
To prove this lemma, it is sufficient to show that the following inequality holds:
\begin{align*}
1-\sum_{i=1}^{n-1}\mu(D^{k}_{i})-\sum_{i=n}^{\infty}\mu(X\setminus S_{i-1}^{C_{M_{i}}}B_{i})-\sum_{i=n}^{\infty}\frac{K_{i}}{L_{i}}-\sum_{i=n}^{\infty}\frac{|k|+B_{n-1}}{L_{i}}\leq \mu(D^{k}_{n})\leq 1-\sum_{i=1}^{n-1}\mu(D^{k}_{i}).
\end{align*}
Indeed, assuming the above inequality holds, we have that 
\begin{align*}
\mu(D^{k}_{n})&\leq 1-\sum_{i=1}^{n-1}\mu(D^{k}_{i})\\
&\leq 1- \bigg(1-\sum_{i=1}^{n-2}\mu(D^{k}_{i})-\sum_{i=n-1}^{\infty}\mu(X\setminus S_{i-1}^{C_{M_{i}}}B_{i})-\sum_{i=n-1}^{\infty}\frac{K_{i}}{L_{i}}-\sum_{i=n-1}^{\infty}\frac{|k|+B_{n-2}}{L_{i}-K_{i}}\bigg)-\sum_{i=1}^{n-2}\mu(D^{k}_{i})\\
&\leq \sum_{i=n-1}^{\infty}\mu(X\setminus S_{i-1}^{C_{M_{i}}}B_{i})+\sum_{i=n-1}^{\infty}\frac{K_{i}}{L_{i}}+\sum_{i=n-1}^{\infty}\frac{|k|+B_{n-2}}{L_{i}-K_{i}}\\
&< \sum_{i=n-1}^{\infty}\epsilon_{i}+\sum_{i=n-1}^{\infty}\epsilon_{i}+\sum_{i=n-1}^{\infty}\frac{|k|+B_{n-2}}{L_{i}-K_{i}}\\&=2\sum_{i=n-1}^{\infty}\epsilon_{i}+\sum_{i=n-1}^{\infty}\frac{|k|+B_{n-2}}{L_{i}}.
\end{align*}
Note that we simply get the upper bound for $\mu(D^{k}_{n})$ by noticing that $D^{k}_{n}\subset X\setminus \bigg(\bigcup_{i=1}^{n-1}D_{i}^{k}\bigg)$. For the lower bound, we first notice that $x\in X$ is guaranteed to be in $D^{k}_{n}$ if both $x$ and $S_{i-1}^{\alpha_{n-1}(x,k)}(x)$ are in the same rigid block of the cocycle $\beta_{i}$ for all $i\geq n$. So, a point $x\in X$ may not be in $D^{k}_{n}$ if it satisfies at least one of the following four conditions:
\begin{enumerate}
\item $x\in \bigcup_{i=1}^{n-1}D_{i}^{k}$,
\item there is an $i\geq n$ such that $x\notin S_{i-1}^{C_{M_{i}}}B_{i}$,

\item there is an $i\geq n$ such that $x$ is in a block, but not a rigid block, of $\beta_{i}$, or
\item there is an $i\geq n$ such that $x$ is in a rigid block of $\beta_{i}$ but $S_{i-1}^{\alpha_{n-1}(x,k)}(x)$ is not. 
\end{enumerate}
So, the measure of the set of points that satisfy each particular condition is bounded above by $\sum_{i=1}^{n-1}\mu(D^{k}_{i})$, $\sum_{i=n}^{\infty}\mu(X\setminus S_{i-1}^{C_{M_{i}}}B_{i})$, $\sum_{i=n}^{\infty}\frac{K_{i}}{L_{i}}$, and $\sum_{i=n}^{\infty}\frac{|k|+B_{n-1}}{L_{i}-K_{i}}$ respectively.
\end{proof}

Now, we wish to show that we can make similar estimates for $\alpha^{-1}$. First, we must show that $\alpha^{-1}$ can be written as a sequence of cocycles.
\begin{lemma} If $\{\alpha_{n}\}_{n\in\N}$ is our sequence of cocyles from before, then $\alpha^{-1}=\lim_{n\to\infty}\alpha_{n}^{-1}$.
\end{lemma}
\begin{proof}
For almost every $x\in X$ and every $n\in\Z$, we can find an $i\in\N$ such that for all $i\geq j$, $\alpha(x,n)=\alpha_{j}(x,n)$ so that $\alpha(x,\alpha_{j}^{-1}(x,n))=\alpha_{j}(x,\alpha_{j}^{-1}(x,n))=n$ which implies that $\alpha^{-1}(x,n)=\alpha_{j}^{-1}(x,n)$ for all $j\geq i$. Hence, we have that $\alpha^{-1}=\lim_{n\to\infty}\alpha_{n}^{-1}$.
\end{proof}
With Lemma 6.4 in hand, we can make the following definition, which is the analog of $D_{n}^{k}$ for $\alpha^{-1}$.
\begin{definition} We say a point $x\in X$ has its $\alpha^{-1}$ \textbf{cocycle value for} $\ell$ \textbf{fixed at stage} $N$ if $N$ is the smallest natural number such that $\alpha_{N}^{-1}(x,\ell)=\alpha^{-1}(x,\ell)$. We will let $E^{\ell}_{N}$ denote the set of all points whose cocycle value for $\ell$ is fixed at stage $N$.
\end{definition}
We can also bound the measure of $E_{n}^{\ell}$ above by a bound similar to the one used for the corresponding sets for $\alpha$.

\begin{lemma}
We have for a given $\ell\in\Z$ and  $n>1$ that \begin{align*}
\mu(E^{\ell}_{n})<2\sum_{i=n-1}^{\infty}\epsilon_{i}+\sum_{i=n-1}^{\infty}\frac{\ell}{L_{i}-K_{i}}.
\end{align*}

\end{lemma}
\begin{proof}
Let $\ell\in\Z$ and $n>1$ be given. Analogous to the proof of  Lemma 6.3, the desired inequality will hold if we can show that the following inequality holds:
\begin{align*}
1-\sum_{i=1}^{n-1}\mu(E^{\ell}_{i})-\sum_{i=n}^{\infty}\mu(X\setminus S_{i-1}^{C_{M_{i}}}B_{i})-\sum_{i=n}^{\infty}\frac{K_{i}}{L_{i}}-\sum_{i=n}^{\infty}\frac{\ell}{L_{i}-K_{i}}\leq \mu(E^{\ell}_{n})\leq 1-\sum_{i=1}^{n-1}\mu(E^{\ell}_{i}).
\end{align*}
First, we can see that the upper bound for $\mu(E^{\ell}_{n})$ holds  by noting that $E_{n}^{\ell}\subset X\setminus \bigg(\cup_{i=1}^{n-1}E_{i}^{\ell}\bigg)$. For the lower bound, note that $x\in X$ may not be in $E_{n}^{\ell}$ if $x$ satisfies at least one of the following four conditions:
\begin{enumerate}
\item $x\in \bigcup_{i=1}^{n-1}E_{i}^{\ell}$,
\item there is an $i\geq n$ such that $x\notin S_{i-1}^{C_{M_{i}}}B_{i}$,

\item there is an $i\geq n$ such that $x$ is in a block, but not a rigid block, of $\beta_{i}$, or
\item there is an $i\geq n$ such that $x$ is in a rigid block of $\beta_{i}$ but $S_{i-1}^{\ell}(x)$ is not. 
\end{enumerate}
So, the measure of the set of points that satisfy each particular condition is bounded above by $\sum_{i=1}^{n-1}\mu(E^{\ell}_{i})$, $\sum_{i=n}^{\infty}\mu(X\setminus S_{i-1}^{C_{M_{i}}}B_{i})$, $\sum_{i=n}^{\infty}\frac{K_{i}}{L_{i}}$, and $\sum_{i=n}^{\infty}\frac{\ell}{L_{i}-K_{i}}$ respectively.
\end{proof}
\begin{proposition}
The orbit equivalence between $S$ and $S^{\alpha^{-1}}$ given by the cocycle $\alpha$ is Shannon.
\end{proposition}
\begin{proof}
First, we will show that the Shannon partitions of the cocycle $\alpha$ have finite Shannon entropy.
Let $k\in\Z$ be given.  For a given $\ell\in\Z$, let $P_{\ell}=\{x\in X\mid \alpha(x,k)=\ell\}$ and let $\mathscr{P}=\{P_{\ell}\mid \ell\in\Z\}$. For every $n\in\N$ and $\ell\in\Z$, set $D^{k,\ell}_{n}=D^{k}_{n}\cap P_{\ell}$ where $D^{k}_{n}$ is the set of points whose $\alpha$ cocycle value for $k$ is fixed at stage $n$. Now, note that for every $\ell\in\Z$ we have that $\{D^{k,\ell}_{n}\}_{n\in\N}$ is a partition  of $P_{\ell}$. 
Setting $F_{n}=\{\alpha_{n}(x,k)\mid x\in X\}$, we can see that $D^{k,\ell}_{n}=\emptyset$ if $\ell\notin F_{n}$. Now, by Proposition 2.2, the Shannon entropy of $\mathscr{P}$ will be finite if we show that $\sum -\mu(D_{n}^{k})\log\frac{\mu(D_{n}^{k})}{|F_{n}|}<\infty$.
By Lemma 6.3, we have that 
\begin{align*}
\mu(D^{k}_{n})<2\sum_{i=n-1}^{\infty}\epsilon_{i}+\sum_{i=n-1}^{\infty}\frac{|k|+B_{n-2}}{L_{i}-K_{i}}<3\sum_{i=n-1}^{\infty}\epsilon_{i}
\end{align*}
for large enough $n$. By our choice of $\epsilon_{n}<2^{-n}$ in our construction, there is an $N$ such that, for all $n>N$, we have that $\mu(D_{n}^{k})<2^{-n+4}$ and $|F_{n}|=K_{n}<2^{n^{3}}$. So,  this implies that 
\begin{align*}
\sum_{n>N} -\mu(D^{k}_{n})\log\frac{\mu(D^{k}_{n})}{|F_{n}|}&\leq \sum_{n>N} -2^{-n+4}\log\frac{2^{-n+4}}{|F_{n}|}\\
&\leq \sum_{n>N}-2^{-n+4}\log 2^{-n^{3}-n+4}\\
&<\infty
\end{align*}
Hence, Proposition 2.2 implies that $\sum_{\ell\in\Z}-\mu(P_{\ell})\log \mu(P_{\ell})<\infty$.\par

Next, we will show that the same is true for Shannon partitions of $\alpha^{-1}$. Let $\ell\in\Z$ be given. For a given $k \in\Z$, let $Q_{k}=\{x\in X\mid \alpha^{-1}(x,\ell)=k\}$ and let $\mathscr{Q}=\{Q_{k}\mid k\in\Z\}$. For every $n\in\N$ and $k\in\Z$, set $E^{\ell, k}_{n}=E^{\ell}_{n}\cap P_{k}$ where $E^{\ell}_{n}$ is the set of points whose $\alpha^{-1}$ cocycle value for $\ell$ is fixed at stage $n$. Again, we wish to use Proposition 2.2. First, we see that for every $k\in\Z$ we have $\{E^{\ell, k}_{n}\}_{n\in\N}$ is a partition  of $Q_{k}$. 
Setting $F_{n}=\{\alpha^{-1}_{n}(x,\ell)\mid x\in X\}$, we can see that $E^{\ell,k}_{n}=\emptyset$ if $\ell\notin F_{n}$. So, we only need to show that $\sum -\mu(E_{n}^{\ell})\log\frac{\mu(E_{n}^{\ell})}{|F_{n}|}<\infty$.
By Lemma 6.6, we have that 
\begin{align*}
\mu(E^{\ell}_{n})<2\sum_{i=n-1}^{\infty}\epsilon_{i}+\sum_{i=n-1}^{\infty}\frac{\ell}{L_{i}-K_{i}}<3\sum_{i=n-1}^{\infty}\epsilon_{i}
\end{align*}
for large enough $n$. By our choice of $\epsilon_{n}<2^{-n}$ in our construction, there is an $N$ such that, for all $n>N$, we have that $\mu(E_{n}^{\ell})<2^{-n+4}$ and $|F_{n}|=K_{n}<2^{n^{3}}$. So,  this implies that 
\begin{align*}
\sum_{n>N} -\mu(E^{\ell}_{n})\log\frac{\mu(E^{\ell}_{n})}{|F_{n}|}&\leq \sum_{n>N} -2^{-n+4}\log\frac{2^{-n+4}}{|F_{n}|}\\
&\leq \sum_{n>N}-2^{-n+4}\log 2^{-n^{3}-n+4}\\
&<\infty
\end{align*}
So, by Proposition 2.2, we can conclude that $\sum_{k\in\Z}-\mu(Q_{k})\log \mu(Q_{k})<\infty$

\end{proof}

\end{section}
\begin{section}{Weak Mixing}
In this section, we will find sequences of natural numbers $\{K_{n}\}_{n\in\N}$ and $\{L_{n}\}_{n\in\N}$ and a sequence of real numbers $\{\epsilon_{n}\}_{n\in\N}$ so that we can construct a sequence of cocycles $\{\beta_{n}\}_{n\in\N}$ of the action $\Z\curvearrowright ^{S}(X,\mu)$ such that $\beta_{n}$ is an $(L_{n}-K_{n}, K_{n},L_{n},\epsilon_{n})$-blocked cocycle and, setting $\alpha_{n}=\beta_{n}\circ \beta_{n-1}\circ\dots\beta_{1}$ for every $n\in\N$, the limit $\alpha(x)=\lim_{n\to\infty}\alpha_{n}(x)$ exists and the action $T=S^{\alpha^{-1}}$ is weak mixing. These parameters will also be chosen in a way such that $\alpha$ is a Shannon orbit equivalence between $S$ and $T$.

\begin{definition} Let $G$ be a countable infinite group and $(X,\mathscr{B},\mu)$ be a probability space. we say an action $G\curvearrowright^{S} X$ is \textbf{weak mixing} if for all finite sets $\mathscr{F}\subset\mathscr{B}$ we have that 

\begin{align*}
\liminf_{g\to\infty}\sum_{A,B\in\mathscr{F}}|\mu(A\cap S^{g}B)-\mu(A)\mu(B)|=0.
\end{align*}
\end{definition}
The following characterization of weak mixing is the one we will use.
\begin{proposition}
Let $G$ be a group and let $G\curvearrowright^{S} X$ be an action over the probability space $(X,\mathscr{B},\mu)$. If there is a sequence $\{g_{n}\}_{n\in\N}\subset G$ such that 
\begin{align*}
\lim_{n\to \infty} \mu(A\cap S^{g_{n}}B)=\mu(A)\mu(B)
\end{align*}
for all pairs of measurable sets $A,B\in X$, then $S$ is weak mixing.
\end{proposition}
\begin{proof}
Let $\{g_{n}\}_{n\in\N}\subset G$ such that
\begin{align*}
\lim_{n\to \infty} \mu(A\cap S^{g_{n}}B)=\mu(A)\mu(B)
\end{align*}
for all pairs of measurable sets $A,B\in X$.
Also, let a finite set $\mathscr{F}\subset \mathscr{B}$ and an $\epsilon>0$ be given. Let $E\subset G$ be a finite set. Let $L$ be a lower bound of the set 
\begin{align*}
\bigg\{\sum_{A,B\in\mathscr{F}}|\mu(A\cap S^{g}B)-\mu(A)\mu(B)|\text{ }\bigg|\text{ } g\in G\setminus E\bigg\}
\end{align*}
Following the hypothesis, we can find an $N\in\N$ such that for all $n>N$ we have that 
\begin{align*}
|\mu(A\cap S^{g_{n}}B)-\mu(A)\mu(B)|<\frac{\epsilon}{2^{|\mathscr{F}|}}
\end{align*}
 for each pair $A,B\in \mathscr{F}$ and $g_{n}\in G\setminus E$. Thus, we have for all $n<N$
 \begin{align*}
 \sum_{A,B\in\mathscr{F}}|\mu(A\cap S^{g_{n}}B)-\mu(A)\mu(B)|<\epsilon,
 \end{align*}
 which implies that $L<\epsilon$. Thus, 
 \begin{align*}
 \inf_{g\in G\setminus E}\bigg\{\sum_{A,B\in\mathscr{F}}|\mu(A\cap S^{g}B)-\mu(A)\mu(B)|\bigg\}<\epsilon.
 \end{align*}
Since $E$ and $\epsilon$ were arbitrary, we have that 
\begin{align*}
\liminf_{g\to\infty}\sum_{A,B\in\mathscr{F}}|\mu(A\cap S^{g}B)-\mu(A)\mu(B)|=0.
\end{align*}
\end{proof}

Let $\{\mathscr{P}_{n}\}_{n\in\N}$ be a sequence of measurable partitions of $X$ such that $\mathscr{P}_{n+1}$ refines $\mathscr{P}_{n}$ for every $n\geq 1$ which generate the sigma algebra of $X$ mod null sets and let $\{\epsilon_{n}\}_{n\in\N}$ be a strictly decreasing sequence of real numbers such that $\epsilon_{n}<\frac{1}{2^{n}}$ for every $n\in\N$. We will now proceed with the construction of our sequence of cocycles $\{\beta_{n}\}_{n\in\N}$ by recursively finding appropriate parameters $K_{n}$, $L_{n}$ and $M_{n}$ at the $n$-th stage so that we can build a Rokhlin tower of height $M_{n}$ whose levels over points in the base will be the blocks for the $(L_{n}-K_{n}, K_{n}, L_{n},\epsilon_{n})$-cocycle $\beta_{n}$. Before we describe the rest of the idea of the construction, we need the following definitions.
\begin{definition}
Let a probability measure preserving action $\Z\curvearrowright^{S}(X,\mu)$, a measurable ordered partition $\mathscr{P}=\{P_{i}\}_{i\in I}$ of $X$, and a finite subset $F\subset X$ be given. If $x\in X$, the $(S,\mathscr{P},F)$\textbf{-name} of $x$ is the indexed set $\{i_{k}\}_{k\in F}\subset I$ such that $S^{n}(x)\in P_{i_{k}}$. If $A\subset X$ such that for every $k\in F$ there is an $i_{k}\in I$ such that $S^{k}A\subset P_{i_{k}}$, we say the $(S,\mathscr{P},F)$\textbf{-name} of $A$ is the indexed set $\{i_{k}\}_{k\in F}\subset I$.\par
In addition, given a set $A\subset X$, we define the \textbf{distribution function} of $\mathscr{P}$ over $A$ to be the function $\dist_{A}(\mathscr{P}): I\to [0,1]$ defined by 
\begin{align*}
\dist_{A}(\mathscr{P})(i)=\frac{\mu(A\cap P_{i})}{\mu(A)}.
\end{align*}
Note that we will drop the subscript of the distribution function if we are looking at the distribution over $X$.
\end{definition}
\begin{remark} Let $A\subset X$, an ordered measurable partition $\mathscr{P}=\{P_{i}\}_{i\in I}$ of $X$ for some set $I$, and a set $\{p_{i}\}_{i\in I}\subset [0,1]$  be given. We define a distance function in the following way:
\begin{align*}
|\dist_{A}(\mathscr{P})-\{p_{i}\}_{i\in I}|=\max_{1\leq i\leq n}\{|\dist_{A}(\mathscr{P})(i)-p_{i}|\}.
\end{align*}
\end{remark}

 We will also guarantee that there is a sequence $\{\ell_{n}\}_{n\in\N}$ of integers such that a certain subsequence of our actions satisfy a progressively stronger with respect to the sequence version of the following approximate mixing condition.
\begin{definition} If $\delta>0$, $\mathscr{P}$ is a partition of $X$, and $\ell\in\N$, we say that an action $S$ is $\delta$\textbf{-mixing relative to} $(\ell,\mathscr{P})$ if  we have that 
\begin{align*}
|\dist(\mathscr{P}\vee S^{-\ell}\mathscr{P})-\dist(\mathscr{P})\times \dist(\mathscr{P})|<\delta.
\end{align*}
\end{definition}
We will also need the following definition.
\begin{definition}
Let a probability measure preserving action $\Z\curvearrowright^{S}(X,\mu)$, measurable ordered partitions $\mathscr{P}=\{P_{i}\}_{i\in I}$ and $\mathscr{Q}=\{Q_{i}\}_{i\in I}$ of $X$ over the same indexing set $I$, a finite subset $F\subset X$, and an $\epsilon>0$ be given. For an $x\in X$ and an $i\in I$, let $a_{i}=\frac{|\{k\in F\mid S^{k}(x)\in P_{i}\}|}{|F|}$. We say that the $\mathscr{P}$\textbf{-distribution of the} $(S,\mathscr{P},F)$\textbf{-name of }$x$\textbf{ is within} $\epsilon$\textbf{ of the distribution of} $\mathscr{Q}$ if
\begin{align*}
|\dist(\mathscr{Q})-\{a_{i}\}_{i\in I}|<\epsilon.
\end{align*}
\end{definition}

\par
We will first construct the tower $S^{C_{M_{1}}}B_{1}$ and the cocycle $\beta_{1}$. First, choose a positive integer $K_{1}$ so that all $x$ in a set $E_{1}\subset X$, $\mu(X\setminus E_{1})<\epsilon_{1}$, have that the $\mathscr{P}_{1}$-distribution of their $(S, \mathscr{P}_{1}, C_{K_{1}})$-name is within $\epsilon_{1}$ of the distribution of $\mathscr{P}_{1}$. Also, choose $L_{1}>\frac{K_{1}}{\epsilon_{1}}$ and $M_{1}>\frac{L_{1}}{\epsilon_{1}}$ such that $M_{1}$ is a multiple of $L_{1}$. Using the strong form of the Rokhlin lemma, we can find a set $B_{1}\subset X$ such that $S^{C_{M_{1}}}B_{1}$ is an $\epsilon_{1}$-tower of $S$ and for all $b\in B_{1}$, we have an $F\subset C_{M_{1}}$ such that $|F|>(1-\epsilon_{1})|C_{M_{1}}|$ and for every $n\in F$ we have $S^{n}b\in E_{1}$. Since $L_{1}$ is a divisor of $M_{1}$, we can write $S^{C_{M_{1}}}B_{1}$ as a disjoint union of the sets $S^{C_{L_{1}}+m}(b)$ for each $b\in B_{1}$ and $m\in L_{1}C_{\frac{M_{1}}{L_{1}}}$.  These will be the blocks for the cocycle $\beta_{1}$.\par

Now we will associate a permutation to each one of these blocks, which we do in the same exact manner as Fieldsteel-Friedman. First, for every $k\in K_{1}$, we choose an $(L_{1}-K_{1},K_{1},L_{1})$-permutation $\pi_{k}$ of $C_{L_{1}}$ which translates $C_{L_{1}-K_{1}}$ by $k$. Next, we construct a family $\Sigma_{1}$ of permutations of $C_{M_{1}}$ by choosing in every possible way a $k\in K_{1}$ for each $m\in L_{1}C_{\frac{M_{1}}{L_{1}}}$ so that $\pi_{k}$ acts on $C_{m+L_{1}}$. Finally, we select a measurable function $\sigma_{1}: B_{1}\to \Sigma_{1}$ so that, for each atom $A$ which belongs to the pure partition of $B_{1}$ with respect to $\mathscr{P}_{1}$, we have that $\sigma_{1}|_{A}$ is uniformly distributed on $\Sigma_{1}$. Now that $\sigma_{1}$ specifies a permutation for each block, we have produced a cocycle $\beta _{1}$ and an action $S_{1}=S^{\beta_{1}^{-1}}$. 

 \par
Now, to define the rest of the cocycles in the sequence we first define the sequence $\{K_{n}\}_{n\in\N}$ by the formula $K_{n+1}=K_{n}+1$ for $n\geq 1$. Assume that we have constructed cocycles $\{\beta_{i}\}_{1\leq i\leq n}$ and $\{\alpha_{i}\}_{1\leq i\leq n}$ which are defined by the formula $\alpha_{i}=\beta_{i}\circ\beta_{i-1}\circ\dots\circ\beta_{1}$, and actions $S_{i}=S^{\alpha_{i}^{-1}}$ whenever $1\leq i\leq n$ so that each $\beta_{i}$ is an $(L_{i}-K_{i},K_{i},L_{i},\epsilon_{i})$-blocked cocycle between $S_{i-1}$ and $S_{i}$ for some $L_{i}$. We also create a subsequence $\{K_{n_{i}}\}_{i\in\N}$ of the $\{K_{n}\}_{n\in\N}$ such that $K_{n_{1}}=K_{1}$ with specific properties which we will detail below. We will now describe how to create $\beta_{n+1}$. Let $K_{n_{j}}$ be the largest value in $\{K_{1},\dots, K_{n}\}$ within our designated subsequence $\{K_{n_{i}}\}_{i\in\N}$. Set $n_{j+1}$ to be the smallest value greater than $n_{j}$  such that, by the ergodic theorem, we can guarantee that there is a set $E_{j+1}$ with $\mu(X\setminus E_{j+1})<\epsilon_{j+1}$ such that the  $\mathscr{P}_{j+1}$-distribution of the $(S_{n_{j}},\mathscr{P}_{j+1},-C_{K_{n_{j+1}}})$-name of each point in $E_{j+1}$ is within $\epsilon_{j+1}$ of the distribution of $\mathscr{P}_{j+1}$. We will choose our parameters $L_{n+1}$ and $M_{n+1}$ and the base $B_{n+1}$ for our tower depending on the following two cases.\bigskip

\emph{Case 1}:  $n+1=n_{j+1}$\bigskip

First, we choose an $L_{n+1}$ large enough such that $\frac{K_{n+1}^{3}}{L_{n+1}}<\epsilon_{n+1}$ and $\frac{n+\sum_{i=1}^{n}L_{i}}{L_{n+1}-K_{n+1}}<\epsilon_{n+1}$. Now, we choose an $M_{n+1}$ such that $M_{n+1}>\frac{L_{n+1}}{\epsilon_{n+1}}$ and $M_{n+1}>M_{n}$, $M_{n+1}$ is a multiple of $L_{n+1}$, and, using the strong form of the Rokhlin lemma and the ergodic theorem, we can find a $B_{n+1}\subset X$ such that 
\begin{enumerate}
\item $(C_{M_{n+1}},B_{n+1})$ forms an $\frac{\epsilon_{n+1}}{K_{n+1}^{2}}$-tower of $S_{n}$,
\item For all $b\in B_{n+1}$ there is a set $F\subset C_{M_{n+1}}$ such that $|F|>(1-\epsilon_{n+1})|C_{M_{n+1}}|$ and for every $k\in F$ we have $S_{n}^{k}b\in E_{j+1}$, and
\item setting $O_{j}=\{x\mid S_{n_{j}}(x)=S_{n}(x)\}$, we have that the $\{O_{j},X\setminus O_{j}\}$-distribution of each $(S_{n}, \{O_{j},X\setminus O_{j}\}, C_{M_{n+1}})$-name for points in $B_{n+1}$ are within $\frac{\epsilon_{n+1}}{K_{n+1}}$ of the distribution of $\{O_{j},X\setminus O_{j}\}$. 
\end{enumerate}
\bigskip

\emph{Case 2}:  $n+1\neq n_{j+1}$\bigskip

 For this case, we choose an $L_{n+1}$ large enough so that $\frac{K_{n_{j+1}}^{3}}{L_{n+1}}<\epsilon_{n+1}$ and $\frac{n+\sum_{i=1}^{n}L_{i}}{L_{n+1}-K_{n+1}}<\epsilon_{n+1}$. Next, we choose an $M_{n+1}$ such that $M_{n+1}>\frac{L_{n+1}}{\epsilon_{n+1}}$, and $M_{n+1}>M_{n}$, and $M_{n+1}$ is a multiple of $L_{n+1}$. We can find a $B_{n+1}\subset X$ such that  $(C_{M_{n+1}},B_{n+1})$ forms an $\frac{\epsilon_{n+1}}{K_{n_{j+1}}^{2}}$-tower of $S_{n}$.\par
In either case, we choose our blocks and blocked permutations to define $\beta_{n+1}$ in exactly the same way. Since $L_{n+1}$ is a divisor of $M_{n+1}$, we can write $S^{C_{M_{n+1}}}B_{n+1}$ as a disjoint union of the sets $S^{C_{L_{n+1}}+m}(b)$ for each $b\in B_{n+1}$ and $m\in L_{n+1}C_{\frac{M_{n+1}}{L_{n+1}}}$.  These will be the blocks for the cocycle $\beta_{n+1}$. Now we will associate a permutation to each one of these blocks, which we do in a manner which is analogous to the construction of $\beta_{1}$. First, for every $k\in K_{n+1}$, we choose an $(L_{n+1}-K_{n+1},K_{n+1},L_{n+1})$-permutation $\pi_{k}$ of $C_{L_{n+1}}$ which translates $C_{L_{n+1}-K_{n+1}}$ by $k$. Next, we construct a family $\Sigma_{n+1}$ of permutations of $C_{M_{n+1}}$ by choosing in every possible way a $k\in K_{n+1}$ for each $m\in L_{n+1}C_{\frac{M_{n+1}}{L_{n+1}}}$ so that $\pi_{k}$ acts on $C_{m+L_{n+1}}$. Finally,    if $n+1=n_{j+1}$, we select a measurable function $\sigma_{n+1}: B_{n+1}\to \Sigma_{n+1}$ so that, for each atom $A$ which belongs to the pure partition of $B_{n+1}$ with respect to $\mathscr{P}_{j+1}\vee \{O_{j},X\setminus O_{j}\}$, we have that $\sigma_{n+1}|_{A}$ is uniformly distributed on $\Sigma_{n+1}$. Now that $\sigma_{n+1}$ specifies a permutation for each block, we have produced a cocycle $\beta _{n+1}$ and an action $S_{n+1}=S_{n}^{\beta_{n+1}^{-1}}$.\par

 \par
We will now show how to get our relative weak mixing property to hold for each $S_{n_{i}}$.
\begin{remark} If $j\in\N$, we have that, by our choice of $\beta_{n_{j}}$, the column levels over each atom A of the pure partition of $B_{n_{j}}$ with respect to $\mathscr{P}_{j}\vee \{O_{j-1},X\setminus O_{j-1}\}$ are permuted according to the block permutations of $\beta_{n_{j}}$ to produce a tower $S_{n_{j}}^{C_{M_{n_{j}}}}\bar{B}_{n_{j}}$ where
\begin{align*}
\bar{B}_{n_{j}}=\{S_{n_{j-1}}^{m}(x)\mid x\in B_{n_{j}}\text{ and } \sigma_{n_{j}}(x,m)=1\}.
\end{align*}
We let $\bar{A}$ be the subset of $\bar{B}_{n_{j}}$ which corresponds to the atom $A$ of $B_{n_{j}}$.
\end{remark}

 First, we will need the following measure theory trick.
 \begin{lemma} Let $\epsilon>0$ and a probability space $(X,\mu)$ be given. Given measurable sets $A\subset X$ and $B\subset X$ of nonzero measure, we can find $\delta_{1},\delta_{2}>0$ such that  $\delta_{1},\delta_{2}<\epsilon$ and we can find a finite measurable partition $\mathscr{A}$ of a set $A'\subset A$ such that $\mu(A\setminus A')<\delta_{1}$ and finite measurable partition of $\mathscr{B}$ of a set $B'\subset B$ such that $\mu(B\setminus B')<\delta_{2}$ such that all sets from both partitions have the same measure. 
\end{lemma} 
\begin{proof}
First, we choose $\delta_{1},\delta_{2}$ such that $0< \delta_{1},\delta_{2}<\min\bigg(\epsilon,\mu(A),\mu(B)\bigg)$ and $\frac{\mu(A)-\delta_{1}}{\mu(B)-\delta_{2}}$ is a rational number. So, we can find positive integers $n$ and $m$ such that 
\begin{align*}
\frac{m}{n}=\frac{\mu(A)-\delta_{1}}{\mu(B)-\delta_{2}}
\end{align*}
Hence, if we partition all but $\delta_{1}$ of $A$ uniformly into $m$ many sets and all but $\delta_{2}$ of $B$ uniformly into $n$ many sets, we've produced two partitions from which any set from either partition will have the same measure.
\end{proof}
\begin{lemma} Let a probability measure preserving action $\Z\curvearrowright^{S}(X,\mu)$, a measurable ordered partition $\mathscr{P}=\{P_{i}\}_{i\in I}$ of $X$, and a finite subset $F\subset X$ be given. if there are disjoint $A,B\subset X$ and an $\epsilon>0$ such that $\mu(A)=\mu(B)$ and  we have that, for some $\{r_{i}\}_{i\in I}$,
\begin{align*}
|\dist_{A}(\mathscr{P})-\{r_{i}\}_{i\in I}|<\epsilon
\end{align*}
and
\begin{align*}
|\dist_{B}(\mathscr{P})-\{r_{i}\}_{i\in I}|<\epsilon
\end{align*}
then
\begin{align*}
|\dist_{A \sqcup B}(\mathscr{P})-\{r_{i}\}_{i\in I}|<\epsilon.
\end{align*}
\end{lemma}
\begin{proof}
Let $i\in I$ be given. Then, we have that
\begin{align*}
\frac{\mu(P_{i}\cap(A\sqcup B))}{\mu(A\sqcup B)}=\frac{\mu(P_{i}\cap A)+ \mu(P_{i}\cap B))}{2\mu(A)}
\end{align*}
Thus,
\begin{align*}
|\frac{\mu(P_{i}\cap(A\sqcup B))}{\mu(A\sqcup B)}-r_{i}|&=|\frac{\mu(P_{i}\cap A)+ \mu(P_{i}\cap B))}{2\mu(A)}-r_{i}|=|\frac{\mu(P_{i}\cap A)}{2\mu(A)}+ \frac{\mu(P_{i}\cap B))}{2\mu(B)}-r_{i}|\\
&\leq\frac{1}{2}|\frac{\mu(P_{i}\cap A)}{\mu(A)} -r_{i}|+\frac{1}{2}|\frac{\mu(P_{i}\cap B)}{\mu(B)} -r_{i}|\\&<\frac{\epsilon}{2}+\frac{\epsilon}{2}=\epsilon.
\end{align*}
Since this holds for every $i\in I$, we have that
\begin{align*}
|\dist_{A \sqcup B}(\mathscr{P})-\{r_{i}\}_{i\in I}|<\epsilon.
\end{align*} 
\end{proof}
Now, we can use these lemmas to prove the following fact. 
\begin{lemma}
Let $i\in\N$ and $\ell\in\Z$ be given. If there is an $\epsilon>0$ such that for every atom $A$ of $B_{n_{i}}$ of the pure partition with respect to $\mathscr{P}_{i}\vee \{O_{j-1},X\setminus O_{j-1}\}$ we have that
\begin{align*}
|\dist_{S_{n_{i}}^{C_{M_{i}}}\bar{A}}\bigg( \mathscr{P}_{i}\vee S_{n_{i}}^{-\ell}\mathscr{P}_{i}\bigg)-\dist(\mathscr{P}_{i})\times \dist(\mathscr{P}_{i})|<\epsilon,
\end{align*}
then 
\begin{align*}
|\dist(\mathscr{P}_{i}\vee S_{n_{i}}^{\ell_{i}}\mathscr{P}_{i})-\dist(\mathscr{P}_{i})\times \dist(\mathscr{P}_{i})|<\epsilon+\epsilon_{n_{i}}.
\end{align*}
\end{lemma}
\begin{proof}
Suppose we are given two atoms $A$ and $A'$ in $B_{n_{i}}$ and a $\delta>0$. Recall our definition of $\bar{A}$ and $\bar{A}'$ from Remark 7.7. By Lemma 7.8 and Lemma 7.9, we see that
\begin{align*}
|\dist_{S_{n_{i}}^{C_{M_{i}}}\bar{A}\cup\bar{A}'}\bigg( \mathscr{P}_{i}\vee S_{n_{i}}^{-\ell}\mathscr{P}_{i}\bigg)-\dist(\mathscr{P}_{i})\times \dist(\mathscr{P}_{i})|<\epsilon+2\delta.
\end{align*}
Since $\delta$ was given arbitrarily, 
\begin{align*}
|\dist_{S_{n_{i}}^{C_{M_{i}}}\bar{A}\cup\bar{A}'}\bigg( \mathscr{P}_{i}\vee S_{n_{i}}^{-\ell}\mathscr{P}_{i}\bigg)-\dist(\mathscr{P}_{i})\times \dist(\mathscr{P}_{i})|<\epsilon.
\end{align*} 
applying this inequality to every pair of atoms of $B_{n_{i}}$, we have that 
\begin{align*}
|\dist_{S_{n_{i}}^{C_{M_{i}}}\bar{B}_{n_{i}}}\bigg( \mathscr{P}_{i}\vee S_{n_{i}}^{-\ell}\mathscr{P}_{i}\bigg)-\dist(\mathscr{P}_{i})\times \dist(\mathscr{P}_{i})|<\epsilon
\end{align*}
Finally, since $\mu(S_{n_{i}-1}^{C_{M_{i}}}B_{n_{i}})>1-\epsilon_{n_{i}}$,
\begin{align*}
|\dist(\mathscr{P}_{i}\vee S_{n_{i}}^{\ell_{i}}\mathscr{P}_{i})-\dist(\mathscr{P}_{i})\times \dist(\mathscr{P}_{i})|<\epsilon+\epsilon_{n_{i}}.
\end{align*}
\end{proof}
We will also need the following lemma.
\begin{lemma} Given the sequence of actions $\{S_{n}\}_{n\in\N}$ defined by $S_{n+1}=S^{\beta^{-1}_{n+1}}_{n}$, the subsequence $\{K_{n_{j}}\}_{j\in\N}$ of  $\{K_{n}\}_{n\in\N}$, and an $\ell\in\Z$, then, for all $j\in\N$ and $m\in\{n_{j}+1,\dots,n_{j+1}\}$,
\begin{align*}
\mu(\{x\in X\mid S^{\ell}_{m}(x)=S_{n_{j}}^{\ell}(x)\})\geq 1-\sum_{i=n_{j}+1}^{m}\bigg(\frac{\epsilon_{i}}{K_{n_{j+1}}^{2}}+\frac{K_{i}}{L_{i}}+\frac{|\ell|}{L_{i}}\bigg).
\end{align*}
\end{lemma}
\begin{proof}
Let $j\in\N$ and $m\in\{n_{j}+1,\dots,n_{j+1}\}$ be given. Setting $\beta=\beta_{m}\circ\beta_{m-1}\circ\dots\circ \beta_{n_{j}+1}$, we notice that $S_{m}=S_{n_{j}}^{\beta^{-1}}$. So, if $x\in X$ such that $S^{\ell}_{m}(x)=S^{\ell}_{n_{j}}(x)$, we must have that 
\begin{align*}
S^{\ell}_{n_{j}}(x)=S^{\ell}_{m}(x)=S_{n_{j}}^{\beta^{-1}(x,\ell)}(x)
\end{align*}
So, we must have that $\beta(x,\ell)=\ell$. Note that this equality may not hold for an $x\in X$ if there is an $i\in\N$ such that $n_{j}+1\leq i\leq m$ and at least one of the following properties holds:
\begin{enumerate}
\item $x\in X\setminus S_{i-1}^{M_{i}}B_{i}$,
\item $x$ is in the filler set of a block, or
\item $x$ is in a rigid block, but $S_{i-1}^{\ell}(x)$ is not. 
\end{enumerate}
We have that, for each $i$ between $n_{j}+1$ and $m$ inclusive, the set of points which satisfy each condition is less than $\frac{\epsilon_{i}}{K_{n_{j+1}}^{2}}$, $\frac{K_{i}}{L_{i}}$, and $\frac{|\ell|}{L_{i}}$ respectively. Thus, removing these sets for every $i$ between $n_{j}+1$ and $m$ yields
\begin{align*}
\mu(\{x\in X\mid S^{\ell}_{m}(x)=S_{n_{j}}^{\ell}(x)\})\geq 1-\sum_{i=n_{j}+1}^{m}\bigg(\frac{\epsilon_{i}}{K_{n_{j+1}}^{2}}+\frac{K_{i}}{L_{i}}+\frac{|\ell|}{L_{i}}\bigg)
\end{align*}
\end{proof}
\begin{lemma} For every $i\in\N$, if $A$ is an atom of the pure partition of $B_{n_{i}}$ respect to $\mathscr{P}_{i}\vee \{O_{i-1},X\setminus O_{i-1}\}$, then for every $b\in A$ all but a proportion of at most $4\epsilon_{n_{i-1}}$ of the $m\in M_{n_{i}}$ have that the $(S_{n_{i}-1}, \mathscr{P}_{i}, -C_{K_{n_{i}}})$-name of $S^{m}_{n_{i}-1}(b)$ agrees with its $(S_{n_{i-1}}, \mathscr{P}_{i}, -C_{K_{n_{i}}})$-name.
\end{lemma}
\begin{proof}
Let an atom $A$ of $B_{n_{i}}$ and a $b\in A$ be given. By Lemma 7.11, 
\begin{align*}
\mu(\{x\in X\mid S_{n_{i}-1}(x)=S_{n_{i-1}}(x)\})\geq 1-\sum_{j=n_{i-1}+1}^{n_{i}-1}\bigg(\frac{\epsilon_{j}}{K_{n_{i}}^{2}}+\frac{K_{j}}{L_{j}}+\frac{1}{L_{j}}\bigg).
\end{align*}
Recall that we chose our tower $S_{n_{i}-1}^{C_{M_{n_{i}}}}B_{n_{i}}$ so that for every $b\in A$ we have that the distribution of the $(S_{n_{i}-1},\{O_{i-1},X\setminus O_{i-1}\}, C_{M_{n_{i}}})$-name of $b$ is within $\frac{\epsilon_{n_{i}}}{K_{n_{i}}}$ of the distribution of $\{O_{i-1},X\setminus O_{i-1}\}$. So, for all $b\in A$ we have, for all but at most a proportion $\sum_{j=n_{i-1}+1}^{n_{i}-1}\bigg(\frac{\epsilon_{j}}{K_{n_{i}}^{2}}+\frac{K_{j}}{L_{j}}+\frac{1}{L_{j}}\bigg)+\frac{\epsilon_{n_{i}}}{K_{n_{i}}}$ of the $m\in M_{n_{i}}$, that $S^{m}_{n_{i}-1}(b)=S^{m}_{n_{i-1}}(b)$. Thus, for all $b\in A$, we have that all but at most a fraction $4\epsilon_{n_{i-1}}$ of the $m\in M_{n_{i}}$ have that the $(S_{n_{i}-1}, \mathscr{P}_{i}, -C_{K_{n_{i}}})$-name of $S^{m}_{n_{i}-1}(b)$ agrees with its $(S_{n_{i-1}}, \mathscr{P}_{i}, -C_{K_{n_{i}}})$-name since
\begin{align*}
K_{n_{i}}\bigg(\sum_{j=n_{i-1}+1}^{n_{i}-1}\bigg(\frac{\epsilon_{j}}{K_{n_{i}}^{2}}+\frac{K_{j}}{L_{j}}+\frac{1}{L_{j}}\bigg)+\frac{\epsilon_{n_{i}}}{K_{n_{i}}}\bigg)&<\epsilon_{n_{i-1}}+\frac{K_{n_{i}}^{3}}{L_{n_{i-1}+1}}+\frac{K_{n_{i}}^{2}}{L_{n_{i-1}+1}}+\epsilon_{n_{i}}\\&<4\epsilon_{n_{i-1}}.
\end{align*}
\end{proof}
\begin{remark} For every $i\in\N$, set $\ell_{i}=L_{n_{i}}$. We will verify our approximate mixing condition for each $S_{n_{i}}$ relative to $\ell_{i}$.
\end{remark}
\begin{definition} Let $i\in\N$ and $A$ be an atom of the pure partition of $B_{n_{i}}$ respect to $\mathscr{P}_{i}\vee \{O_{i-1},X\setminus O_{i-1}\}$. We say that an element $m\in C_{M_{n_{i}}}$ is a \textbf{good position} for $A$ if for all $b\in A$ we have that both $S_{n_{i}-1}^{m-C_{K_{n_{i}}}}(b)$ and $S_{n_{i}-1}^{m+\ell_{i}-C_{K_{n_{i}}}}(b)$ are in rigid blocks of $\beta_{n_{i}}$, and 
\begin{align*}
|\dist(\mathscr{P}_{i})-\dist_{S_{n_{i}-1}^{m-C_{K_{n_{i}}}}(A)}(\mathscr{P}_{i})|<\epsilon_{n_{i}}
\end{align*}
and
\begin{align*}
|\dist(\mathscr{P}_{i})-\dist_{S_{n_{i}-1}^{m+\ell_{i}-C_{K_{n_{i}}}}(A)}(\mathscr{P}_{i})|<\epsilon_{n_{i}}
\end{align*}
both hold.
\end{definition}
\begin{lemma}
Let $i\in\N$ and $A$ be an atom of the pure partition of $B_{n_{i}}$ respect to $\mathscr{P}_{i}\vee \{O_{i-1},X\setminus O_{i-1}\}$. If $m\in C_{M_{n_{i}}}$ is a good position for $A$, then 
\begin{align*}
|\dist_{S_{n_{i}}^{m}\bar{A}}\bigg( \mathscr{P}_{i}\vee S_{n_{i}}^{-\ell_{i}}\mathscr{P}_{i}\bigg)-\dist(\mathscr{P}_{i})\times \dist(\mathscr{P}_{i})|<2\epsilon_{n_{i}}.
\end{align*}
\end{lemma}
\begin{proof}
Let $i\in\N$ and an atom $A$ in $B_{n_{i}}$ be given. Note that we have that $S_{n_{i}-1}^{m-C_{K_{n_{i}}}}A$ and $S_{n_{i}-1}^{m+\ell_{i}-C_{K_{n_{i}}}}A$ are in different rigid blocks since we chose $\ell_{i}=L_{n_{i}}$. First, by the facts that we chose our permutations of $C_{M_{i}}$ to define the block permutations of $\beta_{n_{i}}$ to be uniformly distributed on $A$ and each permutation of a point in a rigid block is just translation by some $k\in K_{n_{i}}$, we can see, setting for $k,k'\in K_{n_{i}}$ \begin{align*}
Q_{k,k'}=\bigg\{b\in\bar{A}\mid \beta_{n_{i}}^{-1}(b,m)=m-k\text{ and } \beta_{n_{i}}^{-1}(b,m+\ell )=m+\ell -k'\bigg\},
\end{align*}
that $\{Q_{k,k'}\}_{k,k'\in K_{n_{i}}}$ is a partition of $\bar{A}$ into sets of uniform measure $\frac{\mu(\bar{A})}{K_{n_{i}}^{2}}$. Now, fix $P,P'\in\mathscr{P}_{i}$. We have that $\mu(S_{n_{i}}^{-m}P\cap S_{n_{i}}^{-(m+\ell_{i})}P'\cap Q_{k,k'})=\mu(Q_{k,k'})$ precisely when $S_{n_{i}-1}^{m-k}A\subset P$ and $S_{n_{i}-1}^{m+\ell_{i}-k'}A\subset P'$. Otherwise, $\mu(S_{n_{i}}^{-m}P\cap S_{n_{i}}^{-(m+\ell_{i})}P'\cap Q_{k,k'})=0$. Thus,
\begin{align*}
\mu\bigg(S_{n_{i}}^{-m}P\cap S_{n_{i}}^{-(m+\ell_{i})}P'\cap \bar{A}\bigg)&=\sum_{k,k'\in K_{n_{i}}}\mu\bigg(S_{n_{i}}^{-m}P\cap S_{n_{i}}^{-(m+\ell_{i})}P'\cap Q_{k,k'}\bigg)\\
&= \frac{|\{k\in K_{n_{i}}\mid S_{n_{i}-1}^{m-k}A\subset P\}|}{K_{n_{i}}}\frac{|\{k'\in K_{n_{i}}\mid S_{n_{i}-1}^{m+\ell_{i}-k'}A\subset P'\}|}{K_{n_{i}}}\mu(\bar{A}).
\end{align*}
Since $m$ is a good position for $A$, we recall that 
\begin{align*}
|\dist(\mathscr{P}_{i})-\dist_{S_{n_{i}-1}^{m-C_{K_{n_{i}}}}(A)}(\mathscr{P}_{i})|<\epsilon_{i}
\end{align*}
and
\begin{align*}
|\dist(\mathscr{P}_{i})-\dist_{S_{n_{i}-1}^{m+\ell_{i}-C_{K_{n_{i}}}}(A)}(\mathscr{P}_{i})|<\epsilon_{i}
\end{align*}
both hold. Hence, we get that 
\begin{align*}
\bigg|\frac{\mu\bigg(S_{n_{i}}^{-m}P\cap S_{n_{i}}^{-(m+\ell_{i})}P'\cap \bar{A}\bigg)}{\mu(\bar{A})}-\mu(P)\mu(P')\bigg|<2\epsilon_{n_{i}}.
\end{align*}
Since our choices of $P$ and $P'$ were arbitrary, we have 
\begin{align*}
|\dist_{S_{n_{i}}^{m}\bar{A}}\bigg( \mathscr{P}_{i}\vee S_{n_{i}}^{-\ell_{i}}\mathscr{P}_{i}\bigg)-\dist(\mathscr{P}_{i})\times \dist(\mathscr{P}_{i})|<2\epsilon_{n_{i}}.
\end{align*}
\end{proof}
\begin{lemma}
Let $i\in\N$ and $A$ be an atom of the pure partition of $B_{n_{i}}$ with respect to $\mathscr{P}_{i}\vee \{O_{i-1},X\setminus O_{i-1}\}$. Then, all but a fraction at most $8\epsilon_{n_{i-1}}+4\epsilon_{n_{i}}$ of the $m\in C_{M_{n_{i}}}$ are good positions for $A$.
\end{lemma}

\begin{proof} 
Let $i\in\N$ and the atom $A$ be given. First, recall that we created a set $E_{i}$ such that $\mu(X\setminus E_{i})<\epsilon_{n_{i}}$ and that $x\in E_{i}$ implies that the $(S_{n_{i-1}},\mathscr{P}_{i}, -C_{K_{n_{i}}})$-name of $x$ has distribution within $\epsilon_{n_{i}}$ of the distribution of $\mathscr{P}_{i}$. Also, we chose our base $B_{n_{i}}$ so that for all $b\in B_{n_{i}}$ there is a set $F\subset C_{M_{i}}$ such that $|F|>(1-\epsilon_{n_{i}})|C_{M_{n_{i}}}|$ and for every $k\in F$ we have that $S_{n_{i}-1}^{k}(b)\in E_{i}$. So, since every element of $b\in A$ has the same $(S_{n_{i-1}},\mathscr{P}_{i}, -C_{M_{n_{i}}})$-name, all but at most a fraction $\epsilon_{n_{i}}$ of the $m\in C_{M_{i}}$ have that
\begin{align*}
|\dist(\mathscr{P}_{i})-\dist_{S_{n_{i-1}}^{m-C_{K_{n_{i}}}}(A)}(\mathscr{P}_{i})|<\epsilon_{n_{i}}.
\end{align*}
Lemma 7.12 implies that all but a proportion of at most $4\epsilon_{n_{i-1}}+\epsilon_{n_{i}}$ of the $m\in C_{M_{i}}$ have that
\begin{align*}
|\dist(\mathscr{P}_{i})-\dist_{S_{n_{i}-1}^{m-C_{K_{n_{i}}}}(A)}(\mathscr{P}_{i})|<\epsilon_{n_{i}}.
\end{align*}
So, the proportion of $m\in C_{M_{i}}$ that satisfy both 
\begin{align*}
|\dist(\mathscr{P}_{i})-\dist_{S_{n_{i}-1}^{m-C_{K_{n_{i}}}}(A)}(\mathscr{P}_{i})|<\epsilon_{n_{i}}
\end{align*}
and
\begin{align*}
|\dist(\mathscr{P}_{i})-\dist_{S_{n_{i}-1}^{m+\ell_{i}-C_{K_{n_{i}}}}(A)}(\mathscr{P}_{i})|<\epsilon_{n_{i}}
\end{align*}
is at most 
$8\epsilon_{n_{i-1}}+2\epsilon_{n_{i}}$.
Now, notice that, since we set $\ell_{i}=L_{n_{i}}$ to be exactly one block length, for all $b\in A$,  $S_{n_{i}-1}^{m-C_{K_{n_{i}}}}(b)$ is not in a rigid block of $\beta_{n_{i}}$ or $S_{n_{i}-1}^{m+\ell_{i}-C_{K_{n_{i}}}}(b)$ is not in a rigid block of $\beta_{n_{i}}$  precisely when $S_{n_{i}-1}^{m}(b)$ occupies one of the first $K_{n_{i}}$ levels of a rigid block or $S_{n_{i}-1}^{m}(b)$ is in the top block of the tower. This accounts for at most a fraction $\frac{K_{n_{i}}}{L_{n_{i}}}+\frac{L_{n_{i}}}{M_{n_{i}}}<2\epsilon_{n_{i}}$ of the $m\in C_{M_{i}}$. Combining this with our previous estimate yields that at least a fraction $1-8\epsilon_{n_{i-1}}+4\epsilon_{n_{i}}$ of the $m\in M_{n_{i}}$ are good positions for $A$.\par

\end{proof}
\begin{proposition} There is a sequence $\{\delta_{i}\}_{i\in\N}$ of non-negative real numbers such that $\lim_{i\to\infty}\delta_{i}=0$ and, for every $i>1$ and for every $j\geq i$, the action $S_{n_{j}}$ is $\delta_{i}$-mixing relative to $(\ell_{i},\mathscr{P}_{i})$.
\end{proposition}
\begin{proof}
Set $\delta_{i}=\frac{3}{2^{n_{i}}}+8\epsilon_{n_{i-1}}+7\epsilon_{n_{i}}$. We will first show that $S_{n_{i}}$ is $\delta_{i}$-mixing relative to $(\ell_{i},\mathscr{P}_{i})$. Let $A$ be an atom of the pure partition of $B_{n_{i}}$ with respect to $\mathscr{P}_{i}\vee \{O_{i-1},X\setminus O_{i-1}\}$. By Lemma 7.15, if $m\in C_{M_{n_{i}}}$ is a good position for $A$, then 
\begin{align*}
|\dist_{S_{n_{i}}^{m}\bar{A}}\bigg( \mathscr{P}_{i}\vee S_{n_{i}}^{-\ell_{i}}\mathscr{P}_{i}\bigg)-\dist(\mathscr{P}_{i})\times \dist(\mathscr{P}_{i})|<2\epsilon_{n_{i}}.
\end{align*}
Combining this with Lemma 7.16 yields that
\begin{align*}
|\dist_{S_{n_{i}}^{C_{M_{i}}}\bar{A}}\bigg( \mathscr{P}_{i}\vee S_{n_{i}}^{-\ell}\mathscr{P}_{i}\bigg)-\dist(\mathscr{P}_{i})\times \dist(\mathscr{P}_{i})|<8\epsilon_{n_{i-1}}+6\epsilon_{n_{i}}.
\end{align*}
Now, we can see that Lemma 7.10
implies that
\begin{align*}
|\dist(\mathscr{P}_{i}\vee S_{n_{i}}^{\ell_{i}}\mathscr{P}_{i})-\dist(\mathscr{P}_{i})\times \dist(\mathscr{P}_{i})|<8\epsilon_{n_{i-1}}+7\epsilon_{n_{i}}<\delta_{i}.
\end{align*}
Hence, $S_{n_{i}}$ is $\delta_{i}$-mixing relative to $(\ell_{i},\mathscr{P}_{i})$.\par

Next, we will prove for every $j> i$, $S_{n_{j}}$ is $\delta_{i}$-mixing relative to $(\ell_{i},\mathscr{P}_{i})$. Following Lemma 7.11, we have that
\begin{align*}
\mu(\{x\mid S_{n_{j}}^{\ell_{i}}(x)\neq S_{n_{i}}^{\ell_{i}}(x)\})\leq \sum_{k=n_{i}+1}^{\infty}\epsilon_{k}+\sum_{k=n_{i}+1}^{\infty}\frac{K_{k}}{L_{k}}+\sum_{k=n_{i}+1}^{\infty}\frac{\ell_{i}}{L_{k}}<\frac{3}{2^{n_{i}}}
\end{align*}
Now, let $P,P'\in \mathscr{P}_{i}$. we can see that
\begin{align*}
|\mu(P\cap S_{n_{j}}^{-\ell_{i}}P')-\mu(P\cap S_{n_{i}}^{-\ell_{i}}P')|&\leq \mu((P\cap S_{n_{j}}^{-\ell_{i}}P')\triangle (P\cap S_{n_{i}}^{-\ell_{i}}P'))\\
&=\mu(P\cap( S_{n_{j}}^{-\ell_{i}}P'\triangle S_{n_{i}}^{-\ell_{i}}P'))\\
&\leq \mu(\{x\mid S_{n_{j}}^{\ell_{i}}(x)\neq S_{n_{i}}^{\ell_{i}}(x)\})\\
&\leq \frac{3}{2^{n_{i}}}
\end{align*}
So, we have that
\begin{align*}
|\dist(\mathscr{P}_{i}\vee S_{n_{j}}^{\ell_{i}}\mathscr{P}_{i})-\dist(\mathscr{P}_{i}\vee S_{n_{i}}^{\ell_{i}}\mathscr{P}_{i})|< \frac{3}{2^{n_{i}}}
\end{align*}
Thus,
\begin{align*}
|\dist(\mathscr{P}_{i}\vee S_{n_{j}}^{\ell_{i}}\mathscr{P}_{i})-\dist(\mathscr{P}_{i})\times \dist(\mathscr{P}_{i})|&\leq |\dist(\mathscr{P}_{i}\vee S_{n_{j}}^{\ell_{i}}\mathscr{P}_{i})-\dist(\mathscr{P}_{i}\vee S_{n_{i}}^{\ell_{i}}\mathscr{P}_{i})|\\
&\text{ }\text{ }\text{ }+|\dist(\mathscr{P}_{i}\vee S_{n_{i}}^{\ell_{i}}\mathscr{P}_{i})-\dist(\mathscr{P}_{i})\times \dist(\mathscr{P}_{i})|\\
&< \frac{3}{2^{n_{i}}}+8\epsilon_{n_{i-1}}+7\epsilon_{n_{i}}\\
&=\delta_{i}.
\end{align*}
\end{proof}
We are now equipped to prove the main theorem.
\begin{theorem}
Suppose $\{\alpha_{n}\}_{n\in\N}$ is a sequence of cocycles of an action $\Z\curvearrowright^{S} (X,\mu)$ with corresponding actions $S_{n}=S^{\alpha_{n}^{-1}}$ for every $n$ which converge to an action $T$.  If there is a sequence $\{\ell_{n}\}_{n\in\N}\subset \Z,$ a sequence $\{\delta_{n}\}_{n\in\N}$ such that $\lim_{n\to\infty}\delta_{n}=0$, and a sequence $\{\mathscr{P}_{n}\}_{n\in\N}$ of measurable partitions of $(X,\mu)$ which are sequentially refining and generate, modulo null sets, the $\sigma$-algebra of $(X,\mu)$ such that, for every $i\in\N$ and $j\geq i$, $S_{n_{j}}$ is weakly $\delta_{i}$-mixing relative to $(\ell_{i},\mathscr{P}_{i})$, we have that $T$ is weakly mixing.
\end{theorem}
\begin{proof}
By Proposition 7.2, we only need to show that for every pair $A,B\subset X$ of non-null measurable sets we have
\begin{align*}
\lim_{i\to \infty}\mu(A\cap T^{-\ell_{i}}B)=\mu(A)\mu(B).
\end{align*}
Let non-null measurable sets $A,B\subset X$ and an $\epsilon>0$ be given. Choose an $N\in\N$ such for all $k>N$:
\begin{enumerate}
\item we can find $j_{1},j_{2}\in\N$ (depending on $k$) such that there are sets $P_{i}\in\mathscr{P}_{n}$ for every $i\in\{1,\dots,j_{1}\}$ and sets $Q_{i}\in \mathscr{P}_{n}$ for $i\in\{1,\dots,j_{2}\}$ such that $\mu\bigg((\cup_{i=1}^{j_{1}}P_{i})\triangle A\bigg)<\frac{\epsilon}{6}$ and $\mu\bigg((\cup_{i=1}^{j_{2}}Q_{i})\triangle B\bigg)<\frac{\epsilon}{6}$.
\item $\delta_{k}<\frac{\epsilon}{6}$.
\end{enumerate}
Fix $k>N$ and set $P=\cup_{i=1}^{j_{1}}P_{i}$ and $Q=\cup_{i=1}^{j_{2}}Q_{i}$. First, note that we can find a $j>k$ such that 
\begin{align*}
\mu\bigg(\{x\mid T^{\ell_{k}}(x)\neq S_{n_{j}}^{\ell_{k}}(x)\}\bigg)<\frac{\epsilon}{6}.
\end{align*}
Hence, 
\begin{align*}
|\mu(P\cap S_{n_{j}}^{-\ell_{k}}Q)-\mu(P\cap T^{-\ell_{k}}Q)|&\leq \mu((P\cap S_{n_{j}}^{-\ell_{k}}P')\triangle (P\cap T^{-\ell_{k}}Q)))\\
&=\mu(P\cap( S_{n_{j}}^{-\ell_{k}}P'\triangle T^{-\ell_{k}}P'))\\
&\leq \mu(\{x\mid S_{n_{j}}^{\ell_{k}}(x)\neq T^{\ell_{k}}(x)\})\\
&\leq \frac{\epsilon}{6}.
\end{align*}
Combining this with Proposition 7.17 and the fact that $\delta_{k}<\frac{\epsilon}{6}$ yields
\begin{align*}
|\mu(P\cap T^{-\ell_{k}}Q)-\mu(P)\mu(Q)|&\leq |\mu(P\cap S_{n_{j}}^{-\ell_{k}}Q)-\mu(P\cap T^{-\ell_{k}}Q)|\\
&\text{ }\text{ }\text{ }+|\mu(P\cap S_{n_{j}}^{-\ell_{k}}Q)-\mu(P)\mu(Q)|\\
&< \delta_{k}+\frac{\epsilon}{6}\\
&<\frac{2\epsilon}{6}.
\end{align*}
The following inequality also holds:
\begin{align*}
|\mu(A\cap T^{-\ell_{k}}B)-\mu(P\cap T^{-\ell_{k}}Q)|&\leq |\mu(A\cap T^{-\ell_{k}}B)-\mu(P\cap T^{-\ell_{k}}B)|\\
&\text{ }\text{ }\text{ }+|\mu(P\cap T^{-\ell_{k}}B)-\mu(P\cap T^{-\ell_{k}}Q)|\\
&\leq |\mu((A\triangle P)\cap T^{-\ell_{k}}B)|\\
&\text{ }\text{ }\text{ }+|\mu(P\cap (T^{-\ell_{k}}B\triangle T^{-\ell_{k}}Q))|\\
&\leq \mu(A\triangle P)+\mu(T^{-\ell_{k}}B\triangle T^{-\ell_{k}}Q)\\
&\leq \mu(A\triangle P)+\mu(B\triangle Q)\\
&<\frac{2\epsilon}{6}
\end{align*}
Finally, putting everything together, we get the following:
\begin{align*}
|\mu(A\cap T^{-\ell_{k}}B)-\mu(A)\mu(B)|&\leq |\mu(A\cap T^{-\ell_{k}}B)-\mu(P\cap T^{-\ell_{k}}Q)|\\
&\text{ }\text{ }\text{ }+|\mu(P\cap T^{-\ell_{k}}Q)-\mu(P)\mu(Q)|\\
&\text{ }\text{ }\text{ }+|\mu(P)\mu(Q)-\mu(A)\mu(B)|\\
&<\frac{2\epsilon}{6}+\frac{2\epsilon}{6}+\frac{2\epsilon}{6}=\epsilon
\end{align*}
Since $k>N$ was arbitrary, we get that $T$ is weakly mixing.
\end{proof}
\end{section}

\bibliographystyle{amsplain}
\bibliography{fincodingsbib}
\end{document}